\documentclass[english]{compositio}
\usepackage[latin1]{inputenc}
\usepackage{psfrag}
\usepackage{mathrsfs,amsthm,amsmath,amssymb}
\usepackage{hyperref}
\theoremstyle{plain}
\newtheorem{thm}{Theorem}[section]
\newtheorem{prop}[thm]{Proposition}
\newtheorem{lem}[thm]{Lemma}
\newtheorem{conj}[thm]{Conjecture}
\newtheorem{cor}[thm]{Corollary}
\theoremstyle{definition}

\newtheorem{exa}[thm]{Example}
\theoremstyle{remark}

\newcommand{\eqdef}{\stackrel{\rm def}{=}}
\usepackage{babel}

\newcommand{\Des}{\mathrm{Des}}

\newcommand{\ZZ}{\mathbb{Z}}
\newcommand{\Z}{\mathbb{Z}}  

\newcommand{\PP}{\mathbb{P}}  
\newcommand{\N}{\mathbb{N}} 
\newcommand{\Q}{\mathbb{Q}}

\newcommand{\qs}{QSym}
\newcommand{\two}{\textrm{\bf {\large 2}}}
\newcommand{\twoindex}{\textrm{\bf {\small 2}}}

\def\({\left(}
\def\){\right)}

\begin{document}
\title[Peak algebras and Kazhdan-Lusztig polynomials]{Peak algebras, 
paths in the Bruhat graph \\
and Kazhdan-Lusztig polynomials\textsuperscript{*} }
\author{Francesco Brenti}
\email{brenti@mat.uniroma2.it}
\address[Francesco Brenti]{Dipartimento di matematica, Universit\`a di Roma "Tor Vergata", Via della Ricerca Scientifica 1, Roma 00133, Italy}
\author{Fabrizio Caselli}
\email{fabrizio.caselli@unibo.it}
\address[Fabrizio Caselli]{Dipartimento di matematica, Universit\`a di Bologna, Piazza di Porta San Donato 5, Bologna 40126, Italy}
\thanks{\textsuperscript{*}Part of this work was carried out while the authors were at the
Mathematisches Forschungsinstitut Oberwolfach as part of their RiP Program.
They would like to thank the MFO for support and for providing an ideal
environment for research.}


\begin{abstract}
We give a new characterization of the peak subalgebra of the algebra of quasisymmetric functions and use this to construct a new basis for this subalgebra. As an application
of these results we obtain a combinatorial formula for the
Kazhdan-Lusztig polynomials which holds in complete generality and
is simpler and more explicit than any existing one. We then show that, 
in a certain sense, this formula cannot be simplified. 
\end{abstract}

\maketitle 
\section{Introduction}

In their seminal paper \cite{K-L} Kazhdan and Lusztig introduced a family of polynomials, indexed 
by pairs of elements of a Coxeter group $W$, that are now known as the Kazhdan-Lusztig polynomials of $W$
(see, e.g., \cite{BjBr} or \cite{Hum}).
These polynomials play a fundamental role in several areas of mathematics, including representation
theory, the geometry of Schubert varieties,
the theory of Verma modules, Macdonald polynomials, canonical bases, immanant 
inequalities, and the Hodge theory of Soergel bimodules (see, e.g., 
\cite{And, BeiBer, BL, BK, EW, FKK, Hai, H, Ha, KL2, Ug}, and the 
references cited there).
Quasisymmetric functions were introduced by Gessel in \cite{Ges} and are related to
many topics in algebra, combinatorics, and geometry including 
descent algebras, Macdonald polynomials, Kazhdan-Lusztig polynomials, enumeration, convex polytopes, noncommutative symmetric functions, Hecke algebras, and Schubert polynomials (see, e.g., 
\cite{As, BayBil, BB, BHV, BH, DKLT, LMW, M-R}, and the references cited there).

In this work we give a new characterization of the peak subalgebra of the algebra of quasisymmetric functions and use this to construct a new basis for this subalgebra with
certain properties. As an application of these results we obtain a 
combinatorial formula for the Kazhdan-Lusztig polynomials of a Coxeter group $W$  
which holds in complete generality and is simpler
and more explicit than any existing one. More precisely, this formula expresses the Kazhdan-Lusztig polynomial of two elements 
$u,v \in W$ as a sum of at most $f_{\ell(u,v)}$ summands ($f_n$ being the $n$-th Fibonacci number), each one of which is the product
 of a number, which depends on $u$ and $v$, and a polynomial, independent of $u,v$, and $W$, and we provide a combinatorial interpretation for both the number and  the polynomial.
We then investigate linear relations between the numbers involved in the formula
and show that there are no ``homogeneous'' relations even for lower intervals of a fixed
rank. A consequence of this result is that the formula that we obtain cannot be simplified by means of linear relations if it is to hold in complete generality. Our proof uses some
new total reflection orderings which may be of independent interest.

The organization of the paper is as follows. In  the next section we collect some notation, definitions, and results
that are needed in the rest of this work. In \S 3 we give a new characterization of the peak subalgebra
of the algebra of quasisymmetric functions (Theorem \ref{dualBB}). 
In \S 4, using this characterization, we construct a basis 
of the peak subalgebra of 
the algebra of quasisymmetric functions with certain properties (Theorem \ref{basis}). In \S 5, using the results in the previous ones, we obtain a combinatorial formula for the 
Kazhdan-Lusztig polynomials which holds in
complete generality (Theorem \ref{finalmain}), and is simpler and more explicit than any existing one. Finally, in \S 6, we study linear relations between the numbers involved in the formula and show, as a consequence of our results
(Corollary \ref{norel}), that the
formula obtained in \S 5 cannot be ``linearly'' simplified.

\section{Preliminaries}

We let $\PP \stackrel{\rm def}{=} \{ 1,2,3, \ldots \} $ ,
$\N \stackrel{\rm def}{=} \PP \cup \{ 0 \} $, 
$\mathbb  Z$ be the ring of integers,  $\mathbb Q$ be the field of rational numbers,
and $\mathbb R$ be the field of real numbers;
for $a \in \N$ we  let $[a] \stackrel{\rm def}{=}
\{ 1,2, \ldots , a \} $ (where $[0] \stackrel{\rm def}{=}
\emptyset $). 
Given $n, m \in \PP$, $n \leq m$, we let
$[n,m]
\stackrel{\rm def}{=} [m] \setminus [n-1]$, and we define
similarly $(n,m]$, $(n,m)$, and $[n,m)$. 
For $S \subseteq \Q$ we 
write $S= \{ a_{1}, \ldots , a_{r} \} _{<}$ to mean that $S= \{ a_{1},
\ldots , a_{r} \}$ and $a_{1} < \cdots < a_{r} $.
The cardinality of a set $A$ will be denoted by  
$|A|$. Given a polynomial $P(q)$, and $i \in \Z$, we
denote by $[q^{i}](P(q))$ the coefficient of $q^{i}$ in $P(q)$. Given $j\in \ZZ$ we let $\chi_{\textrm{odd}}(j)=1$ if $j$ is odd 
and $\chi_{\textrm{odd}}(j)=0$ if $j$ is even, and $\chi_{\textrm{even}}(j)=1-\chi_{\textrm{odd}}(j)$. We let $f_n$ be the $n$-th Fibonacci number defined recursively by $f_0\eqdef 0$, $f_1\eqdef 1$ and $f_{n}\eqdef f_{n-1}+f_{n-2}$ for $n>1$.

Recall that a {\em composition} of
 $n$ ($n \in \PP$) is a 
sequence $(\alpha _{1}, \ldots , \alpha _{s})$ (for some $s \in {\mathbb
P}$) of positive 
integers such that $\alpha _{1}+ \cdots +
\alpha _{s}=n$ (see, e.g., \cite[p. 17]{ECI}). 
 For $n \in \PP$ we let $C_{n}$ be the set of all compositions of
$n$ and $C \stackrel{\rm def}{=} \bigcup _{n \geq 1} C_{n}$.
  Given $\beta \in C$ we denote by $l(\beta )$ the 
number of  parts of $\beta $, by $\beta _{i}$, for $i=1, \ldots ,
l(\beta )$, the $i$-th part of $\beta $ (so that $\beta = (\beta 
_{1}, \beta _{2}, \ldots , \beta _{l(\beta )})$), and we let 
$| \beta | \eqdef \sum _{i=1}^{l(\beta )} \beta _{i}$,
and $T(\beta ) \eqdef \{ \beta _{r}, \beta _{r}+
\beta _{r-1}, \ldots , \beta _{r} + \cdots + \beta _{2} \}$ where
$r \stackrel{\rm def}{=} l(\beta )$. 
 Given $(\alpha_{1},...,\alpha_{s}),(\beta_{1},...,\beta_{t}) \in
C_{n}$ we say that $(\alpha_{1},...,\alpha_{s})$ {\em refines}
$(\beta_{1},...,\beta_{t})$ if there exist $0 < i_{1} < i_{2}
< \cdots < i_{t-1} < s$ such that $\sum_{j=i_{k-1}+1}^{i_{k}}
\alpha_{j} = \beta_{k}$ for $k=1,\ldots,t$ (where $i_{0}
\stackrel{\rm def}{=} 0$
, $i_{t} \stackrel{\rm def}{=} 
s$). We then write $(\alpha_{1},...,\alpha_{s}) \preceq
(\beta_{1},...,\beta_{t})$. It is well known, and 
easy to see, that the map $\alpha 
\mapsto T(\alpha )$ is
an isomorphism from $(C_{n}, \preceq )$ to  
 the Boolean algebra $B_{n-1}$ of subsets of
$[n-1]$, ordered by reverse inclusion.

We let ${\two}\eqdef \{0,1\}$ and for $n\in \mathbb N$ we let ${\two}^n$ be the set of all 0-1 words of length $n$ 
\[
   {\two}^n=\{E=(E_1\cdots E_n):\, E_i\in{\two}\},
\]
$\varepsilon\in \two^0$ be the empty word, and ${\two}^*\eqdef\cup_{n\geq 0}{\two}^n$. We consider on ${\two}^*$ the monoid 
structure given by concatenation. We say that $E\in \two^*$ is \emph{sparse} if either $E=\varepsilon$ or $E$ belongs to the submonoid generated by 0 and 01 and we let $\two^*_s$ be the monoid of sparse sequences. We also let 
$1{\two}^*\eqdef \{1E:\, E\in {\two}^*\}$ and  $\overline{1\two^*}\eqdef 1\two^*\cup \{\epsilon\}$, and we similarly 
define ${\two}^*1$ and $\overline{{\two}^*1}$. If $E\in {\two}^n$ we let $\check E\eqdef (E_1\cdots E_{n-1}(1-E_n))$ if $n\geq 1$,
 and $\check \varepsilon=\varepsilon$, $\overline{E}$ be the complementary string
 (so the $i$-th element of $\overline{E}$ is $1$ if and only if the $i$-th element of $E$ is $0$, for $i\in [n]$), and $E^{op}$ be
 the opposite string of $E$ (so the $i$-th element of $E^{op}$ is $1$ if and only if the $n+1-i$-th element of $E$ is $1$, 
for $i\in [n]$). For notational convenience, for  $1\leq i\leq j$ we also let $E_{i,j}\eqdef 0^{i-1}10^{j-i}$ and $E_{0,j}=0^j$.
 Finally, we let $S(E)\eqdef \{i\in[n]:\, E_i=1\}$, $m_i(E) \eqdef | \{ j \in [n] : E_j = i \} |$,
for $i \in \{ 0,1 \}$, and $\ell(E) \eqdef n$.
We consider on $\two ^n$ the natural partial order $\leq$ defined by $E \leq F$ 
if and only if $S(E) \subseteq S(F)$. 
 
We assume here that the reader is familiar with the basics of the theory of
quasisymmetric functions, for example, as described in \cite[\S 7.19]{ECII}.
We denote by ${\mathcal Q} \subset \Q[[x_{1},x_{2}, \dots ]]$ the algebra of all quasisymmetric functions (with rational coefficients).  
${\mathcal Q}$ is a graded algebra with the
usual grading of power series; we denote by ${\mathcal Q}_{i}$ the
i$^{th}$ homogeneous part of ${\mathcal Q}$, so
${\mathcal Q} = {\mathcal Q}_{0}\oplus {\mathcal Q}_{1} \oplus \cdots .$
 If $E\in \two^{n-1}$ and $S(E)=\{s_1,\ldots,s_t\}_<$ we let $oc(E)=(n-s_t,s_t-s_{t-1},\cdots,s_2-s_1,s_1)$; $oc(E)$ is a composition of $n$ and we
  denote by $M_E$ the monomial quasisymmetric function $M_{oc(E)}$. So, for example, $M_{001010}=
\sum_{1 \leq i_1 < i_2 < i_3} x_{i_1}^{2} x_{i_2}^2 x_{i_3}^3$.
In turn, for $F\in \two^{n-1}$, we let $L_F=\sum_{E \geq F} M_E$ be the fundamental quasisymmetric function.
Note that what we denote $L_E$ is denoted by $L_{oc(E)}$ in \cite[\S 7.19]{ECII}, 
and that the degree of $M_{E}$ and $L_{E}$ is $\ell(E)+1$.

An interesting subalgebra of ${\mathcal Q}$ is the subspace $\Pi$ of \emph{peak functions} 
(see \cite{BHV} and \cite{Ste}). The following result is known (see \cite[Proposition 1.3]{BHV}), and can be taken as the definition of $\Pi$.

\begin{thm}
\label{BBrel}
Let $F = \sum_{E \in \twoindex ^{\ast}} c_E \, M_E \in {\mathcal Q}$.
Then the following are equivalent:
\begin{enumerate}
\item[i)]
$F \in \Pi$;
\item[ii)]
for all $E \in \overline{\two ^{\ast} 1}$, $F \in \overline{1 \, \two ^{\ast}}$, and $j \geq 1$ 
\begin{equation}
\label{}
 \sum _{i=1}^{j} (-1)^{i-1} \, c_{E E_{i,j} F} 
= 2\chi_{\textrm{odd}}(j) c_{E 0^j F}.  
\end{equation}
\end{enumerate}
\end{thm}

The relations in part ii) of the above result are known as the {\em Bayer-Billera}
(or {\em generalized Dehn-Sommerville}) relations (see, e.g., \cite{BayBil}).

 Let $\mathscr V_n$ be the $\mathbb Q$-vector space of functions on ${\two}^n$ taking values in $\mathbb Q$. In particular, $\dim_{\mathbb Q}(\mathscr V_n)=2^n$. If $\alpha \in \mathscr V_n$ and $E\in {\two}^n$ we let $\alpha_E \eqdef \alpha(E)$ be the value that $\alpha$ takes on $E$.

Let $P$ be an Eulerian partially ordered set of rank $n+1$ with minimum $\hat 0$ and maximum $\hat 1$; 
we always assume that a chain ${\mathcal C} = (x_1,\ldots,x_k)$ in $P$ does not
contain $\hat 0$ and $\hat 1$.
Given such a chain we define $E(\mathcal C)\in {\two}^n$ by
\[
   E(\mathcal C)_i=1 \Leftrightarrow \exists j\in[k]:\,\rho(x_j)=i,
\]
where $\rho$ is the rank function of $P$.
The \emph{flag f-vector} of $P$ is the element $f(P)\in \mathscr V_n$ given by  
\[
f(P)_E \eqdef|\{\textrm{chains $\mathcal C$ in $P$}:\,E(\mathcal C)=E\}|   
\]
for all $E\in {\two}^n$.

Let $\mathscr A_n$ be the subspace of $\mathscr V_n$ generated by the flag f-vectors  $f(P)$ of all Eulerian posets of rank $n+1$. The following result is then well known
(see \cite{BayBil}).
\begin{thm}[Bayer-Billera] 
\label{bbrel}
\label{BayBil}
The vector space $\mathscr A_n$  has dimension $f_{n+1}$ and it is determined by the following linear relations: given $\alpha\in \mathscr V_n$ we have $\alpha\in \mathscr A_n$ if and only if for all $E\in \overline{{\two}^*1}$, $F\in \overline{1{\two}^*}$ and $j\geq 1$ such that $E0^jF\in {\two}^n$, we have
\[
\sum_{i=1}^{j}(-1)^{i-1}\alpha_{E\,E_{i,j}\, F}=2\chi_{\textrm{odd}}(j)\alpha_{E\,0^j\, F}.
\]
\end{thm}
Note that these relations are exactly the relations that appear in the characterization of the peak algebra in Theorem \ref{BBrel}.

If $P$ is a graded poset of rank $n+1$, the function $h(P)\in \mathscr V_n$ which is uniquely determined by 
\[
   f(P)_F=\sum_{E\leq F} h(P)_E,
\]
is called the \emph{flag h-vector} of $P$. The definition is clearly equivalent to
\[
   h(P)_F=\sum_{E\leq F}(-1)^{|F-E|}f(P)_E
\]
by the principle of inclusion-exclusion. 

We follow \cite{BjBr} for general Coxeter groups notation
and terminology.
In particular, given
a Coxeter system $(W,S)$ and $u \in W$ we denote by $\ell
(u )$ the length of $u $ in $W$, with respect to $S$,
 by $e$
the identity of $W$, and we let $T \overset{\rm def}{=}
\{ u s u ^{-1} : u \in W, \; s \in S \}$ be the
set of \emph{reflections} of $W$.
We always assume that $W$ is partially ordered by {\em Bruhat
order}. Recall (see, e.g., \cite[\S 2.1]{BjBr}) that this means that
$x \leq y$ if and only if there exist $r \in \N$ and $t_{1},
\ldots , t_{r} \in T$ such that $ t_{r} \cdots  t_{1} \, x=y$ and
$\ell (t_{i} \cdots t_{1} \, x) > \ell(t_{i-1} \cdots t_{1}x)$ for $i=1,
\ldots ,r$.
Given $u,v \in W$ we let $[u,v] \overset{\rm def}{=} \{ x \in W
: u \leq x \leq v \}$.  We consider $[u,v]$ as a poset with the 
partial ordering induced by $W$.
It is well known (see, {\it e.g.}, 
\cite[Corollary 2.7.11]{BjBr})
that intervals of $W$ are Eulerian
posets.
Recall (see, e.g., \cite[\S 2.1]{BjBr})
that the {\em Bruhat graph} of a Coxeter system $(W,S)$
is the directed graph $B(W,S)$ obtained by taking $W$ as vertex
set and putting a directed edge from $x$ to $tx$ for all $x \in W$ and
$t \in T$ such that $\ell(x) < \ell(tx)$. We denote by $\Phi ^{+}$
 the set of positive roots of $(W,S)$ (see, e.g., \cite[\S 4.4]{BjBr}). Recall
(see, e.g., \cite[\S 5.2]{BjBr}) that a total ordering $\prec$ on $\Phi^+$ is a \emph{reflection ordering} if 
whenever $\alpha,\beta, c_1\alpha+c_2\beta\in \Phi^+$ for some $c_1,c_2\in \mathbb R_{>0}$ 
and $\alpha \prec \beta$ then $\alpha \prec c_1\alpha+c_2\beta \prec \beta$. The existence of reflection orderings 
(and many of their properties) is proved in \cite[\S 2]{Dy} (see also \cite[\S 5.2]{BjBr}).
 By means of the canonical bijection between $\Phi^+$ and $T$ 
(see, e.g., \cite[\S 4.4]{BjBr}) we transfer the reflection ordering also on $T$. Given $u,v \in W$ we denote 
by $P_{u,v}(q)$ the Kazhdan-Lusztig polynomial 
of $u,v$ in $W$ (see, e.g., \cite[Chap. 5]{BjBr} and \cite[Chap. 7]{Hum}).
 
Let $\prec$ be a reflection ordering of $T$. 
Given a path $\Delta = (a_{0}, a_{1} , \ldots ,
a_{r})$  in $B(W,S)$ from $a_{0}$ to $a_{r}$, we define
 its {\em length} to be $l(\Delta) \overset{\rm def}{=} r$, and its
 {\em descent string} with respect to $\prec$ to be the sequence $E_\prec (\Delta )\in \two^{r-1}$ given by
\[
E_ \prec (\Delta)_{r-i}=1 \Leftrightarrow 
a_{i}(a_{i-1})^{-1} \succ a_{i+1} (a_{i})^{-1}. 
\]

Given $u,v \in W$, and $k \in {\mathbb N}$,
 we denote by $B_{k}(u,v)$ the set of all the directed
paths in $B(W,S)$ from $u$ to $v$ of length $k$, and we let $B(u,v)
\overset{\rm def}{=} \bigcup _{k \geq 0} B_{k}(u,v)$.
For $u,v \in W$, and $E \in \two^{n-1}$,
we let, following \cite{BreLMS},
\begin{equation}
\label{}
c(u,v)_E \overset{\rm def}{=}
| \{ \Delta  \in B_{n}(u,v) : \;
E_ \prec (\Delta ) \leq E \} | , 
\end{equation}
and
\begin{equation}
\label{balpha}
 b(u,v)_E \overset{\rm def}{=}
| \{ \Delta  \in B_{n}(u,v): \; E_ \prec (\Delta )
= E \} | .
\end{equation}
 Note that these definitions  imply that
\begin{equation}
\label{2.IE.CB}
 c(u,v)_E = \sum _{\{ F \in \twoindex^{n-1} : \, F \leq E\} }
b(u,v)_F 
\end{equation}
for all $u,v \in W$ and $E \in \two^{n-1}$.
It follows immediately from Proposition 4.4 of \cite{BreLMS} that $c(u,v)_E$ (and hence $b(u,v)_E$) are
independent of the reflection ordering $\prec $ used to define them.

Let $[u,v]$ be a Bruhat interval  of rank $r+1$ in a Coxeter group, and 
$\Delta = (x_0,x_1,\ldots,x_{n+1})$ a path in the Bruhat graph from $u$ to $v$.
So $x_0=u$, $x_{n+1}=v$, and for all $i\in[n+1]$ we have $x_{i-1}<x_i$ and the element $t_i$ given by $x_{i}=x_{i-1}t_i$ is a reflection. We then sometimes denote such a path by
\[
   \Delta=(x_0 \stackrel{t_1}{\longrightarrow}x_1 \stackrel{t_2}{\longrightarrow}\cdots \stackrel{t_{n+1}}{\longrightarrow}x_{n+1}).
\]

If $\Delta\in B_{n+1}(u,v)$, $\prec$ is a reflection ordering and $E=E_{\prec}(\Delta)$ we let $m_\prec(\Delta)\eqdef \mu_{E_n}\cdots \mu_{E_1}\in \mathbb Z\langle a,b\rangle$, where $\mu_0=a$ and $\mu_1=b$. In other words, if $\Delta=(x_0 \stackrel{t_1}{\longrightarrow}x_1 \stackrel{t_2}{\longrightarrow}\cdots \stackrel{t_{n+1}}{\longrightarrow}x_{n+1})$, then $m_\prec(\Delta)$ is the product of $n$ factors, the $i$-th factor being $a$ if $t_i\prec t_{i+1}$ and $b$ otherwise. We will usually drop the subscript $\prec$ from the notation $m_\prec(\Delta)$ when it is clear from the context.

If $[u,v]$ is a Bruhat interval of rank $r+1$ the {\em $cd$-index} of $[u,v]$ is the polynomial
\[
   \Psi_{[u,v]}\eqdef \sum_{E\in \twoindex ^r}h([u,v])_E\mu_E\in \mathbb Z\langle a,b\rangle,
\]
where $h([u,v])$ is the flag $h$-vector of $[u,v]$ and $\mu_E\eqdef \mu_{E_1}\cdots \mu_{E_r}$. It is known that $\Psi_{[u,v]}$ 
is a polynomial in $c=a+b$ and $d=ab+ba$, as $[u,v]$ is a Eulerian poset. It is also known that if $[u,v]$ has rank $r+1$ 
then there exists a unique path $\Delta\in B_{r+1}(u,v)$ such that $E_\prec(\Delta)=0^r\in \two^r$. This implies 
(see \cite[Theorem 3.13.2]{ECI}) that for all $E\in \two^r$ we have that $b([u,v])_{E^{op}}=h([u,v])_E$. 
Therefore the $cd$-index of a Bruhat interval $[u,v]$ of length $r+1$ can be expressed as
\[
   \Psi_{[u,v]}=\sum_{\Delta\in B_{r+1}(u,v)}m_\prec(\Delta),
\]
where $\prec$ is any reflection ordering.
We consider the natural extension of this polynomial to all paths in the Bruhat graph
\[
   \tilde \Psi_{[u,v]}(a,b)\eqdef \sum_{\Delta\in B(u,v)}m_\prec(\Delta).
\]
The polynomial $\tilde \Psi_{[u,v]}$ has been introduced by Billera and the first author in \cite{BB} and can also 
be expressed as a polynomial in the variables $c=a+b$ and $d=ab+ba$ and therefore it is called the {\em complete $cd$-index} 
of the interval $[u,v]$. We will use the simpler notation   $\tilde \Psi_{u,v}$ instead of $\tilde \Psi_{[u,v]}$ to 
denote the complete $cd$-index of the Bruhat interval $[u,v]$.

Let $\mathcal A=\mathbb Z\langle a,b\rangle$. Following \cite{ER}, we define  a coproduct $\delta:\mathcal A\rightarrow \mathcal A\otimes \mathcal A$ on $\mathcal A$  as the unique linear map such that for all $n\in \mathbb N$ and all $v_1,\ldots, v_n\in \{a,b\}$,
\[
 \delta(v_1\cdots v_n)=\sum_{i=1}^n v_1\cdots v_{i-1}\otimes v_{i+1}\cdots v_n.
\]
One can observe that the algebra $\mathcal A$ endowed with the coproduct $\delta$ has also a Newtonian coalgebra structure, though this is not needed in the sequel.

Now let $\mathcal P$ be the $k$-vector space consisting of formal finite linear combinations of Bruhat intervals. 
We define also on $\mathcal P$ a coproduct $\delta:\mathcal P \rightarrow \mathcal P\otimes \mathcal P$ in the following way.
 We let
\[
 \delta([u,v])=\sum_{x\in (u,v)}[u,x]\otimes[x,v].
\]
The following result is proved in \cite[Proposition 2.11]{BB}
\begin{prop}\label{coalgmap}
 The complete $cd$-index $\tilde \Psi:\mathcal P\rightarrow \mathcal A$ is a coalgebra map, i.e.
\[
 \sum_{x\in (u,v)}\tilde \Psi_{u,x}\otimes \tilde \Psi_{x,v}=
\delta(\tilde \Psi_{u,v}).
\]
\end{prop}

For all $x\in \mathcal A$ write $\delta(x)=\sum_i x_i(1)\otimes x_i(2)$ where $x_i(1),x_i(2)\in \mathcal A$. Then for any $y\in \mathcal A$ we can consider the following map
\[
 D_y(x)=\sum_i x_i(1)\cdot y \cdot x_i(2).
\]
One can easily verify that this is a well-defined linear map, and that it is a derivation, i.e. it satisfies the Leibniz rule on products, for all $y\in \mathcal A$. The following is then an immediate consequence of Proposition \ref{coalgmap}.

\begin{cor}\label{der}
Let $[u,v]$ be any Bruhat interval. Then
\[
 D_y(\tilde \Psi_{u,v})=\sum_{x\in (u,v)}\tilde \Psi_{u,x}\cdot y \cdot  \tilde \Psi_{x,v}.
\]
\end{cor}

Given $u,v \in W$, $u \leq v$, we let, following \cite{BB},
\[
\widetilde{F}(u,v) \eqdef \sum_{E \in \, \twoindex ^{\ast}} b(u,v)_E L_E,
\]
if $u<v$ and $\widetilde{F}(u,u) \eqdef 1$.
This definition is different from the one given in \cite{BB} but
equivalent to it by
Theorem 2.2 of \cite{BB}. The following is proved in 
\cite[Theorem 2.2]{BB} (see also \cite[Theorem 8.4]{BreJAMS}).
\begin{thm}
\label{FinPi}
Let $(W,S)$ be a Coxeter system and $u,v \in W$, $u<v$. Then 
$\widetilde{F}(u,v) \in \Pi$.
\end{thm}

Let $n\in \mathbb N$. A \emph{lattice path} of length $n$ is a function $\Gamma : [0,n] \rightarrow {\mathbb Z}$ such that
 $\Gamma (0) =0$ and
 $ | \Gamma (i) - \Gamma (i-1)  |=1 $ for all $ i \in [n]$ and we denote by ${\mathcal L}(n)$ the set
 of all the lattice paths of length $n$. Given  $\Gamma \in \mathcal L(n)$ we let $N(\Gamma)\in \two^{n-1}$ be given by
 \[ N(\Gamma )_i=1 \Longleftrightarrow \Gamma (i)<0,
\]
and $d_{+} (\Gamma ) \stackrel{\rm def}{=} | \{ i \in [n]: \Gamma
 (i) - \Gamma (i-1) =1 \} |$.
 Note that 
\begin{equation}
\label{dplus}
 d_{+}(\Gamma ) = \frac{\Gamma (n)+n}{2} . 
\end{equation}

For $E \in \two^{n-1}$ we define, following \cite[\S 5.4]{BjBr}, a polynomial $\Upsilon_E(q)\in \mathbb Z[q]$ by
\begin{equation}
\label{Upsdef}
\Upsilon_{E}(q) = (-1)^{m_0(E)}\sum_{\Gamma \, \in \, {\mathcal L}(E)}
(-q)^{d_{+}(\Gamma)},
\end{equation}
where ${\mathcal L}(E) \eqdef \{ \Gamma \in {\mathcal L}(n) : N(\Gamma)=E \}$. 
 Note that what we denote $\Upsilon _{E}$ is denoted by $\Upsilon _{oc (E)}$ in
 \cite[\S 5.4]{BjBr}. For example,
$\Upsilon_{001010}(q)=q^4-q^3$.

Following \cite{BB} we define a linear map 
${\mathcal K} : \qs \rightarrow \Z[q^{1/2},q^{-1/2}]$ by
${\mathcal K}(L_E) \eqdef q^{- \frac{\ell(E)+1}{2}} \Upsilon_E$, for all 
$E \in \two ^{\ast}$.
We then have the following result (see \cite[Proposition 3.1]{BB}, and
\cite[Theorem 5.5.7]{BjBr}).
\begin{thm}
\label{FtoKL}
Let $u,v \in W$, $u<v$.  Then
\begin{equation}
{\mathcal K}(\widetilde{F}(u,v)) =
q^{\frac{-\ell(u,v)}{2}} P_{u,v}(q) - 
q^{\frac{\ell(u,v)}{2}} P_{u,v} \left( \frac{1}{q} \right) . 
\end{equation}
\end{thm}

Given $E \in \two ^{\ast}$ we let the {\em exponent composition} of $E$ be the unique 
 composition $\alpha =(\alpha _{1},\alpha _{2}, \ldots )$  such that
 \[
 E= \left\{ \begin{array}{ll} 
 \underbrace{1 \ldots 1}_{\alpha _{1}} \, \underbrace{0 \ldots 0}_{\alpha _{2}} \,
 \underbrace{1 \ldots 1}_{\alpha _{3}} \ldots , & \mbox{if $E_{1}=1$,} \\
 \underbrace{0 \ldots 0}_{\alpha _{1}} \, \underbrace{1 \ldots 1}_{\alpha _{2}} \,
 \underbrace{0 \ldots 0}_{\alpha _{3}} \ldots , & \mbox{if $E_{1}=0$.} 
 \end{array} \right. \]
 So, for example, the exponent composition of $00110$ is $(2,2,1)$.
The following is a restatement of Corollary 6.7 of \cite{BreJAMS}. Note that 
$\Upsilon_E \neq 0$ if the exponent composition of $E$ has only one
part.

\begin{cor}
\label{upsilonzero}
Let $E \in \two^{\ast}$ be such that $\ell(\alpha) \geq 2$, where $\alpha$ 
is the exponent composition of $E$. Then 
$\Upsilon_{E} \neq 0$ if and only if $\alpha_2 \equiv \alpha_3 \equiv \cdots \equiv \alpha_{\ell(\alpha)-1} \equiv 1$ (mod 2) and 
$\alpha_1 \equiv E_1$ (mod 2). $\Box$
\end{cor}

\section{A characterization of the peak algebra}\label{betarela}

Our purpose in this section is to give a new characterization of the peak
subalgebra of the algebra of quasisymmetric functions. More precisely, we give 
necessary and sufficient conditions on the coefficients of a quasisymmetric 
function $F$, when expressed as a linear combination of fundamental quasisymmetric
functions, for $F$ to be in the peak subalgebra. Our result is the following.

\begin{thm}
\label{dualBB}
Let $F = \sum_{E \in \twoindex ^{\ast}} \beta_E \, L_E \in {\mathcal Q}$.
Then the following are equivalent:
\begin{enumerate}
\item[i)]
$F \in \Pi$;
\item[ii)]
for all $E,F\in {\two}^*$ 
\[
\beta_{EF}+\beta_{\check EF}= \beta_{E\bar F}+\beta_{\check E \bar F}.
\]
\end{enumerate}
\end{thm}

The rest of this section is devoted to the proof of this result, and to some 
consequences of it. 

Let $\mathscr B_n$ be the vector subspace of $\mathscr V_n$ generated by the flag h-vectors of all Eulerian posets of rank $n+1$. Then $\mathscr B_n$ has clearly dimension $f_{n+1}$ by 
Theorem \ref{BayBil}, and the result that we wish to prove is clearly equivalent to 
the following one.

\begin{thm}
\label{h-relations}
 Let $\alpha,\beta\in \mathscr V_n$ be such that $\alpha_F=\sum_{E\leq F} \beta_E$ for all $F\in {\two}^n$. Then the following are equivalent
\begin{itemize}  
\item $\alpha$ satisfies Bayer-Billera relations (i.e. $\alpha\in \mathscr A_n$);
\item for all $E,F\in {\two}^*$ such that $EF\in {\two}^n$
\begin{equation}\label{dualrel}\beta_{EF}+\beta_{\check EF}= \beta_{E\bar F}+\beta_{\check E \bar F}.
\end{equation}
\end{itemize}
\end{thm}

We refer to the relations appearing in \eqref{dualrel} as the \emph{dual Bayer Billera-relations}.
   Note that the relations $\beta_F=\beta_{\bar F}$ appear as a special case of \eqref{dualrel} by letting $E=\epsilon$. 
Let $\mathscr B'_n$ be the subspace of $\mathscr V_n$ defined by the relations \eqref{dualrel}. 
The proof of Theorem \ref{h-relations} is a consequence of the following facts which are proved in Propositions \ref{alphainan} and \ref{dimbn} respectively:
\begin{enumerate}
   \item if $\beta\in \mathscr B'_n$ and $\alpha \in \mathscr V_n$ is given by $\alpha_F=\sum_{E\leq F}\beta_E$, then $\alpha\in \mathscr A_n$.
\item $\dim(\mathscr B'_n) \geq f_{n+1}$.
\end{enumerate}

\begin{lem}
\label{solvingones}
If $\beta \in \mathscr B'_n$  then for all $E,F\in {\two}^*$ and $j>0$ such that $E1^jF\in {\two}^n$ we have
\[
   \beta_{E1^jF}=\sum_{i=1}^{j-1}(-1)^{i-1}\beta_{EE_{i,j}\bar F}+(-1)^{j-1}\beta_{EE_{j,j}F}+\chi_{\mathrm{even}}(j)\beta_{E0^j\bar F}.
\]
\end{lem}
\begin{proof}
   We proceed by induction on $j$. If $j=1$ the relation is a trivial identity and so we assume $j\geq 2$. We then have, using Eq. \eqref{dualrel} and our induction hypothesis, that 
\begin{align*}
   \beta_{E 1^j F}&=
\beta_{E 10^{j-1}\bar  F} +\beta_{E 0^j\bar  F}- \beta_{E 01^{j-1} F}\\
&=
\beta_{E E_{1,j}\bar  F} +\beta_{E 0^j\bar  F}-\sum_{i=1}^{j-2}(-1)^{i-1}\beta_{E 0 E_{i,j-1}\bar  F}+(-1)^{j-1} \beta_{E 0 E_{j-1,j-1} F}\\&\hspace{5mm}-\chi_{\textrm{even}}(j-1)\beta_{E 0^{j}\bar  F}\\
&=\sum_{i=1}^{j-1}(-1)^{i-1}\beta_{EE_{i,j}\bar F}+(-1)^{j-1}\beta_{EE_{j,j}F}+\chi_{\textrm{even}}(j)\beta_{E0^j\bar F}.
\end{align*}
\end{proof}

\begin{lem}\label{relga}
If $\beta\in \mathscr B'_n$ then for all $E\in {\two}^*$ and $F\in \overline{{1\bf 2}^*}$ such that $EF\in {\two}^n$, we have
\[
   \sum_{F'\leq F}\beta_{EF'}=\sum_{F'\leq F}\beta_{E\overline{F'}}.
\]

\end{lem}
\begin{proof}
If $F=\varepsilon$ the result is trivial, so we can assume $F=1 G$ for some $G\in {\two}^*$. Then
\begin{align*}
   \sum_{F' \leq F}\beta_{E F'}&=
   \sum_{G' \leq G}(\beta_{E0G'}+\beta_{E1G'})
\end{align*}
while
\[
   \sum_{F'\leq F }\beta_{E \overline{F'}}=
   \sum_{G'\leq G }(\beta_{E\overline{0G'}}+\beta_{E\overline{1G'}})=
   \sum_{G'\leq G}(\beta_{E 1 \overline{G'}}+\beta_{E 0 \overline{G'}}),
\]
and the result follows from \eqref{dualrel}.
\end{proof}
Using the relation $\beta_E=\beta_{\bar E}$, one can obtain in an analogous way the following symmetric version of 
Lemma \ref{relga}: for all $\beta\in \mathscr B'_n$, $E\in\overline{{\two}^*1}$ and $F\in {\two}^*$ such that $EF\in {\two}^n$ we have
\[
   \sum_{E'\leq E}\beta_{E'F}=\sum_{E'\leq E}\beta_{\overline {E'} {F}}.
\]
\begin{prop}\label{alphainan}
   Let $\beta\in \mathscr B'_n$ and $\alpha\in V_n$ be given by 
\[
   \alpha_F=\sum_{E\leq F}\beta_E
\]
for all $F\in {\two}^n$. Then $\alpha\in \mathscr A_n$.
\end{prop}
\begin{proof}
  Let  $j \geq 1$, $E\in \overline{{\two}^*1}$, $F\in \overline{1{\two}^*}$ be such that $E0^jF\in {\two}^n$. We have to show that 
\begin{equation}\label{bbo}
  2\chi_{\textrm{odd}}(j)\alpha_{E 0^j F}=\sum_{i=1}^{j}(-1)^{i-1} \alpha_{E E_{i,j} F}. 
\end{equation} 
By the definition of $\alpha$, Eq. \eqref{bbo} is equivalent to
\[
2 \chi_{\textrm{odd}}(j)\sum_{E'\leq E,\,F'\leq  F}\beta_{E'0^jF'}=   \sum_{E'\leq E,\,F'\leq F}\sum_{i=1}^{j}(-1)^{i-1}\beta_{E'  E_{i,j} F'}.
\]
By Lemma \ref{solvingones} this reduces to 
\begin{align*}
2 \chi&_{\textrm{odd}}(j) \sum_{E'\leq E,\,F'\leq F}\beta_{E'0^jF'} =\\& 
\sum_{E'\leq E,\,F'\leq F}\Big(\beta_{E' 1^{j}\overline{F'}}+(-1)^{j-1}\beta_{E' E_{j,j} F'}+(-1)^j\beta_{E' E_{j,j}\overline{F'}}-\chi_{\textrm{even}}(j)\beta_{E'0^jF'}\Big)
\end{align*}
and using the relations $\beta_{E' 1^j\overline {F'}}=\beta_{\overline{E'}0^jF'}$ and $\beta_{E'E_{j,j}F'}+\beta_{E'0^jF'}=\beta_{E'E_{j,j} \overline{F'}}+\beta_{E'0^j\overline{F'}}$ to conclude the proof we only have to verify that
\[
  2\chi_{\textrm{odd}}(j)\sum_{E'\leq E,\,F'\leq F} \beta_{E'0^j{F'}}=\sum_{E'\leq E,\,F'\leq F}(\beta_{\overline{E'}0^jF'}+(-1)^{j-1}\beta_{{E'}0^jF'});
\]
but this is an immediate consequence of the symmetric version of Lemma \ref{relga}.
\end{proof}

\begin{prop}\label{dimbn}
   We have $\dim(\mathscr B'_n)\geq f_{n+1}$.
\end{prop}
\begin{proof}
The result follows if we show that there exists a vector space $\mathscr W_n$ of dimension $f_{n+1}$ and an injective linear map 
$\beta:\mathscr W_n \rightarrow \mathscr B'_n$. The vector space $\mathscr W_n$ is defined as follows. Let $\mathbb Q\langle a,b\rangle $ be the ring of polynomials with rational coefficients in two noncommuting variables $a,b$. 
Then $\mathscr W_n$ is the subspace of  $\mathbb Q\langle a,b\rangle $ consisting of all homogeneous polynomials of degree 
$n$ that can be expressed as polynomials in $a+b$ and $ab+ba$. It is known and not difficult to check that the monomials 
of the form $(a+b)^{m_0}(ab+ba)\cdots (a+b)^{m_{r-1}}(ab+ba)(a+b)^{m_r}$, with $m_0+\cdots+m_r+2r=n$, form a basis for $\mathscr W_n$. Since there are $f_{n+1}$ such monomials we have that $\dim \mathscr W_n=f_{n+1}$. The map $\beta$ is defined as follows. 
Recall that we let $\mu_0=a$, $\mu_1=b$ and for $E\in \two ^n$ we let $\mu_E=\mu_{E_1}\cdots \mu_{E_n}$. It is clear that $\{\mu_E:\, E\in \two ^n\}$ is a basis for 
the vector space 
of homogeneous polynomials in $\mathbb Q\langle a,b \rangle$ of degree $n$. Therefore, if $P\in \mathscr W_n$ we have $P=\sum_{E\in \twoindex ^n}\beta(P)_E \mu_E$ for some  $\beta(P)\in \mathscr V_n$. 
To prove that $\beta(P)\in \mathscr B_n'$ for all $P\in \mathscr W_n$ we can clearly assume that $P$ is one of the basis elements shown before. We proceed by induction on $\deg P$. If $\deg P=0$ the result is trivial. 
   If $\deg P>0$ then either $P=P'(a+b)$ for some $P'\in \mathscr W_{n-1}$ or $\deg P>1 $ and $P=P''(ab+ba)$ for some $P''\in \mathscr W_{n-2}$.  In the last case we have $\beta(P)_{T01}=\beta(P)_{T10}=\beta(P'')_T$ and $\beta(P)_{T00}=\beta(P)_{T11}=0$ for all $T\in \two^{n-2}$. We have to show that $\beta(P)$ satisfies the dual Bayer-Billera relations, i.e.
\begin{equation}\label{dure}
   \beta(P)_{EF}+\beta(P)_{\check EF}=\beta(P)_{E\bar F}+\beta(P)_{\check E \bar F},
\end{equation}
for all $E,F\in \two^*$ such that $EF\in \two^n$.
If $F=\epsilon$ Eq. \eqref{dure} is trivial. If $F\in \two^1$ we can clearly assume $E_{n-1}=F=0$. 
We have $\beta(P)_{EF}=\beta(P)_{\check E \bar F}=0$, and, if $E=G0$,  $\beta(P)_{\check E F}=\beta(P)_{G10}=\beta(P)_{G01}=\beta(P)_{E\bar F}$ and Eq. \eqref{dure} follows. 
So we can assume that $F\in \two ^m$ with $m\geq 2$. If $F_{m-1}=F_{m}$ Eq. \eqref{dure} is trivial and for symmetry reasons we can assume $F=G01$ for some $G\in \two ^{m-2}$. 
We have
\[
   \beta(P)_{EF}+\beta(P)_{\check EF}=\beta(P)_{EG01}+\beta(P)_{\check EG01}=\beta(P'')_{EG}+\beta(P'')_{\check E G},
\]
while
\[
   \beta(P)_{E\bar F}+\beta(P)_{\check E\bar F}=\beta(P)_{E\bar G10}+\beta(P)_{\check E\bar G10}=\beta(P'')_{E\bar G}+\beta(P'')_{\check E \bar G},
\]
and the result follows by the inductive hypothesis applied to $P''$. 

In the first case $\beta(P)_{E0}=\beta(P)_{E1}=\beta(P')_E$ for all $E\in \two^{n-1}$ and Eq. \eqref{dure} similarly follows from our inductive hypothesis on $P'$. 
The map $\iota$ is clearly injective, and the proof is complete. 
\end{proof}

As a corollary of our result we have the following characterization of $\mathbb Q\langle a+b, ab+ba \rangle$ as a subspace of $\mathbb Q\langle a, b \rangle$.
\begin{cor}
\label{cd}
Let $\beta\in \mathscr V_n$. Then the polynomial
\[
P(a,b)=\sum_{E\in {\twoindex}^n}\beta_E\mu_E\in \mathbb Z \langle a, b \rangle,
\]
can be expressed as a polynomial in $a+b,ab+ba$ if and only if $\beta$ satisfies Eq. \eqref{dualrel}.
\end{cor}

Note that Corollary \ref{cd}, together with Theorem \ref{h-relations}, give a different
proof of the existence of the $cd$-index of Eulerian posets (see \cite{BayKla}, or \cite[Theorem 3.17.1]{ECI}).

\section{A basis for the peak algebra}

In this section we define a family of quasisymmetric functions and show, using
the results in the previous one, that the ones that are nonzero are a basis for 
the peak subalgebra of the algebra of quasisymmetric functions. 
We also show how to expand any peak quasisymmetric function as a linear 
combination of elements of this basis. These results 
are used in the next section in the proof of our main result.

For $E \in \two ^{n-1}$ let 
$\partial(E) \eqdef \{ i \in [n-2] : E_{i} \neq E_{i+1} \} \cup \{ n-1 \}$.
Note that $\partial(E)=\{ x_1,\ldots,x_{r} \}_{<}$ if and only if the exponent
composition of $E$ is $(x_1,x_2-x_1,x_3-x_2,\ldots,x_{r}-x_{r-1})$.
Let $T \in \two ^{n-1}$, $S(T)=\{ s_{1},\ldots ,s_{t} \}_{<}$, 
$s_{0} \eqdef 0$, $s_{t+1} \eqdef n$. We let  ${\mathcal G} (T)$ be the
set of all $E=\two ^{n-1}$ such that  
\begin{description}
\item[i)] $\partial (E) \cap (s_{j},s_{j+1})\neq \emptyset$ for all $j \in [0,t-1]$;
\item[ii)] if $x,y\in \partial (E) \cap (s_{j},s_{j+1})$ then $x\equiv y \pmod 2$ for all $j \in [0,t]$;
\end{description}
Given such an $E$ we define $sgn (E,T) \eqdef (-1) ^{\sum_{j=0}^{t-1}
(s_{j+1}-y_{j}-1)}$ where $y_j$ is any element of 
$\partial (E) \cap (s_{j},s_{j+1})$ for $j \in [0,t-1]$, 
and we let
\[ 
D _{T} \eqdef \sum_{E \in {\mathcal G}(T)} sgn(E,T) \, L_{E}\in \mathcal Q . 
\]
So, for example, if $T=00100$ then ${\mathcal G}(T) = \{ 01111, 01100, 00111, 00100, 
10000,$ $10011, 11000,$ $11011 \}$ and $D_T =-L_{01111}-L_{01100}+L_{00100}-L_{10000}-L_{10011}+L_{00111}+L_{11000}+L_{11011}$.
Note that $D_T$ is homogeneous of degree $\ell(T)+1$ and
that ${\mathcal G} (T) = \emptyset$, and hence $D _{T} =0$,  if $T$
is not sparse.
Given $E,T \in \two ^{n-1}$ we let
\[ 
h_{E,T} \eqdef [L_{E}](D _{T}), 
\]
so, by our definitions,
\begin{equation}
h_{E,\, T} = 
\left\{ \begin{array}{ll}
sgn(E,T), & \mbox{if $E \in {\mathcal G}(T)$,} \\
0, & \mbox{otherwise.}
\end{array}
\right. 
\end{equation}
Note that, since $ h_{E,T}$ depends only on $\partial (E) \setminus T$,
given $S \subseteq [n-1]$ we will sometimes write $h_{S,T}$
rather than $h_{E,T}$ if $\partial (E)=S$.

The next property is crucial in the proof of the main result of
this section.
\begin{prop}
\label{peakrel}
Let $E,T \in \two^{n-1}$, and $i \in [2,n-2]$ be such that
$i-1,i \not \in \partial (E)$. Then 
\begin{equation}
\label{peak1}
h_{\partial (E),\, T}+h_{\partial (E) \cup \{ i-1,i \},\, T}=
h_{\partial (E) \cup \{ i \}, \, T} +h_{\partial (E) \cup \{ i-1 \},\, T}. 
\end{equation}
\end{prop}
{\bf Proof:} 
We may clearly assume that $T$ is sparse. 
If $i \in T$ then $\partial (E) \setminus T = 
(\partial (E) \cup \{ i \}) \setminus T$ and 
$(\partial (E) \cup \{ i-1 \}) \setminus T = 
(\partial (E) \cup \{ i-1,i \}) \setminus T$
so (\ref{peak1}) clearly holds. Similarly if $i-1 \in T$.
We may therefore assume that $i, i-1 \notin T$. 
Suppose first that $E \in {\mathcal G}(T)$. Then 
$\partial (E) \cup \{ i-1,i \} \notin {\mathcal G}(T)$ while exactly
one of $\partial (E) \cup \{i-1 \}$, $\partial (E) \cup \{ i \}$
is in ${\mathcal G}(T)$, and it is easy to see that it has the same
sign as $E$, so (\ref{peak1}) holds. Suppose now that $E \notin {\mathcal G}(T)$. 
Then either there is $j \in [0,t-1]$ such that $\partial (E) \cap (s_{j},s_{j+1})
= \emptyset$ or there exists $j \in [0,t]$ such that 
$\partial (E) \cap (s_{j},s_{j+1})=\{ x_{1}, \ldots , x_{p}\} _{<}$
and there exists $r \in [2,p]$ such that $x_{r} - x_{r-1} \equiv 1 \pmod{2}$.
If either $i<s_{j}$ or $s_{j+1}<i-1$ then $\partial (E) \cup \{i-1 \},
\partial (E) \cup \{ i \}, \partial (E) \cup \{ i-1,i \} \notin {\mathcal G}(T)$ 
so (\ref{peak1})
holds. So assume $s_{j}<i-1<i<s_{j+1}$. Suppose first that $\partial (E)
\cap (s_{j},s_{j+1}) = \emptyset$ for some $j \in [0,t-1]$. Then $\partial (E) \cup \{ i-1,i \} \notin {\mathcal G}(T)$ while either both or none of
$\partial (E) \cup \{ i-1 \}$, $\partial (E) \cup \{ i \}$ are in ${\mathcal G}(T)$ and in the
first case $sgn (\partial (E) \cup \{ i-1 \} , T)=-sgn (\partial (E) \cup \{ i \} , T)$
so (\ref{peak1}) holds. Suppose now that 
$\partial (E) \cap (s_{j},s_{j+1}) = \{ x_{1},
\dots , x_{p} \} _{<}$ for some $j \in [0,t]$ and there exists $r \in [2,p]$ 
such that $x_{r}-x_{r-1} \equiv 1 \pmod{2}$. 
Then $\partial (E) \cup \{ i-1,i \}, \partial (E)
\cup \{ i-1 \}, \partial (E) \cup \{ i\} \notin {\mathcal G}(T)$ and 
(\ref{peak1}) holds. 
$\Box$

We can now prove the first main result of this section, namely that the quasisymmetric functions $D_T$ are in the peak
subalgebra of the algebra of quasisymmetric functions.
\begin{thm}
 Let $T \in \two^{n-1}$. Then $D _{T} \in \Pi _n$.
 \end{thm}
 {\bf Proof.} 
Note first that, since
$h_{E,T}$ depends only on $\partial (E) \setminus T$, $h_{E,T}=h_{\bar{E},T}$ 
for all $E \in \two^{n-1}$.
Now let $i \in [2,n-2]$, $A \in \two^{i-1}$, and $E \in \two^{n-1-i}$.
We claim that then
\begin{equation}
\label{peak2}
h_{A0E,\, T}+h_{A1E, \, T}=h_{A0\bar{E}, \, T}+h_{A1\bar{E}, \, T}.
\end{equation}
In fact, we may clearly assume that $E_{1} =0$. Let 
$\{ x_{1}, \ldots , x_{r} \} _{<} \eqdef \partial (A)$ and 
$\{ y_{1} , \ldots , y_{k} \} _{<} \eqdef \partial (E)$ (so $x_{r}=i-1$ and $y_{k}=n-1-i$). If $A_{i-1} =0$, 
then $\partial (A0E) = \{ x_{1}, \ldots ,x_{r-1},y_{1}+i,\ldots ,y_{k}+i \}$, $\partial (A1E)=
 \{ x_{1} , \ldots ,x_{r},i,y_{1}+i, \ldots , y_{k}+i  \}$, $\partial (A0 \bar{E}
 ) = \{ x_{1}, \ldots ,x_{r-1},i,y_{1}+i,\ldots , y_{k}+i \}$,
 and $\partial (A1\bar{E})=\{ x_{1}, \ldots , x_{r}, y_{1}+i,\ldots y_{k}+i \}$ so (\ref{peak2}) follows from Proposition \ref{peakrel}. 
 Similarly, if $A_{i-1} = 1$ then we have that $\partial
 (A1\bar{E}) = \{ x_{1}, \ldots , x_{r-1},y_{1}+i,\ldots , y_{k}+i \}$, $\partial
 (A0\bar{E})=\partial (A1\bar{E}) \cup \{ x_{r},i \}$, $\partial (A0E)=\partial (A1\bar{E})
 \cup \{ x_{r} \}$, and $\partial (A1E)=\partial (A1\bar{E}) \cup \{ i\}$ 
 and (\ref{peak2}) again follows from Proposition \ref{peakrel}. 
This shows that the function $\beta$ given by $\beta_E=h_{E,T}$ belongs to $\mathscr B_{n-1}$. 
This implies the result, by Theorems \ref{BBrel} and \ref{h-relations}.
$\Box$
 
 Let $E \in \two^{n-1}$, $\{ y_{1}, \ldots y_{k} \} _{<} \eqdef \partial (E)$.
 Note that $E$ is sparse if and only if $E_{1}=0$ and 
\begin{equation}
\label{sparsef}
y_{2i}=y_{2i-1}+1
\end{equation}
 for all $1\leq i \leq \left\lfloor \frac{k}{2} \right\rfloor$.
 
The following is the main result of this section. 

\begin{thm} 
\label{basis}
The set $\{ D _{T} :\, T \in \two_s^{\ast} \}\cup\{1\}$ is a basis of $\Pi$. Furthermore, if 
$F = \sum_{E \in \twoindex ^{\ast}} h_E L_E \in \Pi$, then 
$F = \sum_{ E \in \twoindex^{\ast}_s } h_E D_E$.
\end{thm}
{\bf Proof.} Let $T \in \two ^{n-1}_s$, $T =\{ s_{1},\ldots ,s_{t} \}_{<}$, $s_{0} \eqdef 0$, $s_{t+1} \eqdef n$. We claim that
\begin{equation}
\label{delta}
 h_{E,T} =\delta _{E,T}. 
\end{equation}
for all $E \in \two ^{n-1}_s$.
Note first that $\partial (T) \cap (s_{j},s_{j+1})=\{ s_{j+1}-1 \}$
for all $j \in [0,t-1]$, while $|\partial (T) \cap (s_{t},s_{t+1})| \leq 1$. 
Hence $T \in {\mathcal G}(T)$ and $sgn (T,T)=1$ so $h_{T, T}=1$.
Suppose now that $E \in {\mathcal G}(T)$, $E$ sparse, 
$\{ y_{1}, \ldots ,y_{k} \} _{<} \eqdef \partial (E)$. We claim that $\partial (E)
\cap (s_{j-1},s_{j}) =\{ s_{j}-1 \}$, and that $s_{j}=y_{2j}$ for all $j \in [t]$.
 In fact, since, by (\ref{sparsef}), $y_{2}-y_{1}=1$, this is clear if $j=1$. Suppose that it is true
for some $j \in [t-1]$. Then $s_{j}=y_{2j}$. Furthermore, 
$|\partial (E) \cap (s_{j},s_{j+1})|=1$ (for
if $|\partial (E) \cap (s_{j},s_{j+1})| \geq 2$ then 
$y_{2j+1},y_{2j+2} \in \partial (E) \cap (s_{j},s_{j+1})$ which, by
 (\ref{sparsef}), contradicts the fact that $E \in {\mathcal G}(T)$). 
Hence $s_{j}<y_{2j+1}<s_{j+1} \leq y_{2j+2}$
which, by (\ref{sparsef}), implies that $y_{2j+1}=s_{j+1}-1=y_{2j+2}-1$. 
If $\partial (E) \cap (s_{t},s_{t+1}) = \emptyset $ then $s_t = n-1$, so $k=2t$ and
$\partial(E)= \{ s_1 -1, s_1, \ldots, s_t -1, s_t \}$ so $E=T$. If 
$\partial (E) \cap (s_{t},s_{t+1}) \neq \emptyset$ then, since $s_t = y_{2t}$,
$ y_{2t+1} \in \partial (T) \cap (s_{t},s_{t+1})$. But, by (\ref{sparsef}), this implies that 
$y_{2t+1} = n-1$. Hence $k=2t+1$ and
$\partial(E)= \{ s_1 -1, s_1, \ldots, s_t -1, s_t, n-1 \}$ which again implies
that $E=T$.
This shows that 
$\mathcal G(T)\cap \two^*_s = \{ T \}$ 
and hence proves (\ref{delta}). 
Therefore $\{ D _{T}:\, T\in \two^{n-1}_s \}$ 
is a linearly independent set and this, since 
$\dim(\Pi _n) = | \two^{n-1}_s|$, 
proves the result.
$\Box$

\section{Kazhdan-Lusztig polynomials}
In this section, using the results in the two previous ones, we prove 
a nonrecursive combinatorial formula for the Kazhdan-Lusztig polynomials which holds in
complete generality, and which is simpler and more explicit than any existing one.

Let $n\in \PP$, $T\in \two^{n-1}_s$, $S(T) \eqdef \{ s_1,\ldots,s_t \}_{<}$, $s_0 \eqdef 0$,
$s_{t+1} \eqdef n$. We say that a lattice path $\Gamma$ is a {\em $T$-slalom} 
(the reader may want to consult Figure \ref{slalomfield} (top) for an illustration where $n$ is odd, and Figure \ref{slalomfield} (bottom) for an illustration where $n$ is even) if and only if
\begin{itemize}
\item $\ell(\Gamma)=n$;
\item $\Gamma(s_i +1)\neq 0$ for all $i\in[t]$ (i.e. $\Gamma$ does not passes through the ``stars'' in the examples in Figure \ref{slalomfield});
\item $\Gamma$ crosses the segment  $\{y=-\frac{1}{2}, x\in[s_{i-1}+1,s_{i}]\}$ (the dotted segments in Figure \ref{slalomfield}) exactly once for all $i\in [t]$;
\item $\Gamma(x)\geq \chi_{\textrm{even}}(n)$ for all $x>s_t +1$ (i.e. the path $\Gamma$ remains above the solid segment in Figure \ref{slalomfield}).
\end{itemize}
We denote by $\mathcal{SL}(T)$ the set of $T$-slaloms.
For $T\in\two^{n-1}_s$ we let
\[
 \Omega_T(q)\eqdef (-1)^{s_1+\cdots+s_t +t}\sum_{\Gamma\in \mathcal {SL}(T)}(-q)^{d_-(\Gamma)},
\]
where $d_-(\Gamma)=n-d_+(\Gamma)$ is the number of down-steps of $\Gamma$. 
For example, if $T=00100$  there are exactly three paths in $\mathcal{SL}(T)$ (see Figure \ref{slalompaths}) and $ \Omega_{00100}(q)=-q+2q^2$.

We can now state the main result of this section.

\begin{thm}\label{finalmain}
 Let $(W,S)$ be a Coxeter system, $u,v\in W$, $u\leq v$ and $\ell=\ell(v)-\ell(u)$. Then
 \[
  P_{u,v}(q) = \sum_{T\in \twoindex^*_s}b(u,v)_T \, q^{\frac{\ell-\ell(T)-1}{2}} \, \Omega_T(q).
\]
\end{thm}

The rest of this section is devoted to the proof of Theorem \ref{finalmain}.

\begin{figure}
   \includegraphics[scale=.5]{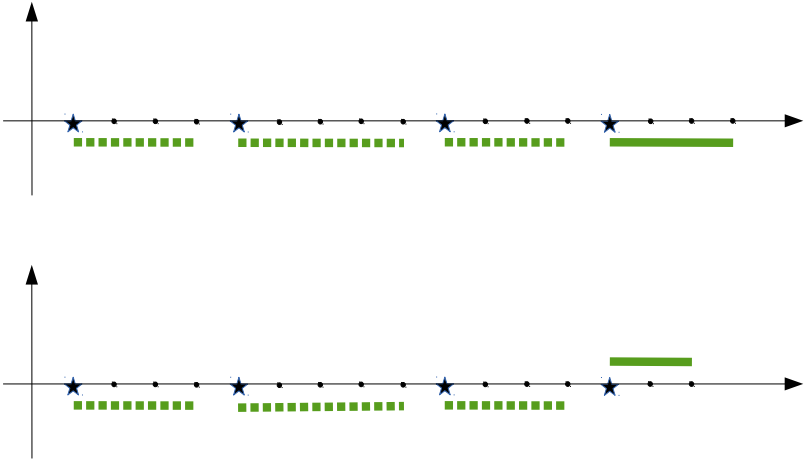}
\caption{The constraints of a slalom path in ${\mathcal {SL}}(T)$, where  $T=0001000010001000$ (top) and  $T=000100001000100$ (bottom).}
\label{slalomfield}
\end{figure}

Let $T \in \two ^{n-1}$, $S(T) \eqdef \{ s_1,\ldots,s_t \}_{<}$, $s_0 \eqdef 0$,
$s_{t+1} \eqdef n$. We define ${\mathcal J} (T)$ to be the set of all 
$E \in \two ^{n-1}$ such that:
\begin{description}
\item[i)] $|\partial(E) \cap (s_j,s_{j+1})|=1$ for all $j\in [0,t-1]$;
\item[ii)] $|\partial(E) \cap (s_t,s_{t+1})| \leq 2$;
\item[iii)] if $\partial(E) \cap (s_t,s_{t+1}) = \{ x, n-1 \}$ then
$x \equiv n-1$ (mod 2).
\end{description}
Given such an $E$ we define $sgn(E,T) \eqdef 
(-1)^{\sum_{i=1}^{t}(s_i - x_i -1)}$ where 
$\{ x_i \} \eqdef \partial(E) \cap (s_{i-1},s_{i})$ for $i \in [t]$ and let
\[
\tilde \Omega_T(q) \eqdef \sum_{E \in \, {\mathcal J}(T)} sgn(E,T) \, \Upsilon_E(q).
\]
We also set ${\mathcal J}(\varepsilon) \eqdef \{ \varepsilon \}$ and
$\tilde \Omega_{\varepsilon}(q) \eqdef \Upsilon_{\varepsilon}(q)$. 
\begin{exa}\label{exa00100} If $T=00100$, then 
${\mathcal J}(T)=\{ 01111,01100,00111,00100,$ $10000,$ $10011,$ $11000,11011 \}$
and $\tilde \Omega_T(q) = - \Upsilon_{01111}- \Upsilon_{01100}
+ \Upsilon_{00111}+ \Upsilon_{00100}- \Upsilon_{10000}- \Upsilon_{10011}+
 \Upsilon_{11000}+ \Upsilon_{11011}=
\Upsilon_{00111}+\Upsilon_{00100}-\Upsilon_{10000}=q^{5}-2q^{4}+2q^2-q$.
\end{exa}
Note that ${\mathcal J}(T)=\emptyset $ if $T$ is not sparse, and that  
${\mathcal J}(T) \subseteq {\mathcal G}(T)$.

\begin{prop}
\label{KLOmega}
Let $(W,S)$ be a Coxeter system and $u,v \in W$, $u<v$.  Then
\begin{equation*}
P_{u,v}(q) - q^{\ell(u,v)} P_{u,v} \left( \frac{1}{q} \right) =
\sum _{T \in \twoindex ^{\ast}} 
q^{\frac{\ell(u,v)- \ell(T)- 1}{2}} \, b(u,v)_T \, \tilde \Omega_{T}(q).
\end{equation*}
\end{prop}
{\bf Proof.} 
Note first that, since $u < v$, $\widetilde{F}(u,v)$ has no constant term.
Hence from Theorems \ref{FinPi} and \ref{basis} we have that
\[ 
\widetilde{F}(u,v)  = \sum _{ T \in \twoindex_s ^{\ast } }
b(u,v)_T D_{T} . 
\]
Applying the linear map ${\mathcal K}$ to this equality we get, by Theorem \ref{FtoKL}, that
\[ 
q^{-\frac{\ell(u,v)}{2}} P_{u,v}(q) -q^{\frac{\ell (u,v)}{2}}P_{u,v} \left( \frac{1}{q} \right) = \sum _{ T \in \twoindex_s ^{\ast } }
b(u,v)_T {\mathcal K} (D_{T}) . 
\]
But, by our definitions, we have that
\begin{equation}
\label{KD}
{\mathcal K}(D_{T}) 
 =  \sum_{E \in {\mathcal G}(T)} sgn (E,T) \, {\mathcal K} (L_{E}) \\
 =  \sum_{E \in {\mathcal G}(T)} sgn (E,T) \, q^{-\frac{\ell(E)+1}{2}} \Upsilon_{E}(q).
\end{equation}

\begin{figure}
   \includegraphics[scale=.5]{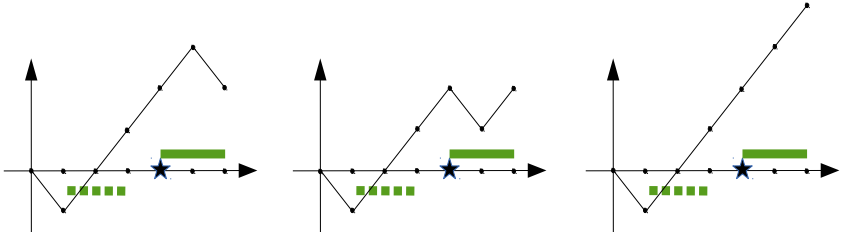}
\caption{Slalom paths associated to $T=00100$}
\label{slalompaths}
\end{figure}

Recall that ${\mathcal J }(T) \subseteq {\mathcal G}(T)$. Let $E \in {\mathcal G}(T)
\setminus {\mathcal J}(T)$, $\{ y_1, \ldots, y_k \}_{<} \eqdef \partial(E)$. 
Then either $|\partial (E) \cap (s_{t},s_{t+1})|
\geq 3$ or $| \partial (E) \cap (s_{j},s_{j+1})| \geq 2$ for some $j \in
[0,t-1]$. But if  either of these conditions hold then $k \geq 3$ and there exists $j \in [k-2]$ such that $y_{j+1} \equiv y_{j} \pmod{2}$ and this, by 
Corollary \ref{upsilonzero},
implies that $\Upsilon_{E}=0$. Hence we conclude from (\ref{KD}) that
\[ 
{\mathcal K} (D_{T}) 
= \sum _{E \in {\mathcal J}(T)} sgn (E,T) \, q^{-\frac{\ell(E)+1}{2}} \, \Upsilon _{E}(q)
= q^{-\frac{\ell(T)+1}{2}} \tilde \Omega_T (q), 
\]
and the result follows. $\Box$

Let $T\in \two^{n-1}_s$ be a sparse sequence of length $n-1$, $s_0=0$ and $S(T)=\{s_1,\ldots,s_t\}_<$. 
We let $\mathcal L(T)$ be the set of all lattice paths $\Gamma$ of length $n$ such that $N(\Gamma)\in \mathcal J(T)$. 
For $\Gamma\in \mathcal L(T)$ and $h\in[t]$ we let $x_h(\Gamma)$ be the unique element in $\partial (N(\Gamma))\cap (s_{h-1},s_{h})$ and $\varepsilon_T(\Gamma)=\sum_{h=1}^t (s_h-x_h(\Gamma)-1)$. We also let $\eta(\Gamma)=m_0(N(\Gamma))=| \{ j\in [n-1]:\, \Gamma(j)\geq 0 \} |$. 
We will usually write $\varepsilon (\Gamma)$ instead of $\varepsilon_T(\Gamma)$ when the sparse sequence $T$ is clear from the context.
\begin{exa}
   Let $T=0010001000$, so $t=2$, $n-1=10$, $S(T)=\{s_1,s_2\}$ with $s_1=3$ and $s_2=7$. The definition of $\mathcal L(T)$ implies that a lattice path belongs to  $\mathcal L(T)$ if and only if it has length $n=11$, it crosses the two dotted  segments in Figure \ref{figurepath} exactly once, and crosses the solid-dotted  segment at most once, but only form NW to SE. 
The path $\Gamma$ depicted in Figure \ref{figurepath} therefore belongs to  $\mathcal L(T)$. In this case $N(\Gamma)=0010111011$ and one can easily check that $x_1(\Gamma)=2$ and  $x_2(\Gamma)=4$. 
Finally, we can observe that $\partial (N(\Gamma))\cap (s_t,s_{t+1})=\{8,10\}$. Hence we have $\varepsilon(\Gamma)=(s_1-x_1(\Gamma)-1)+(s_2-x_2(\Gamma)-1)=0+2=2$. Moreover, we have $\eta(\Gamma)=4$ and $d_+(\Gamma)=5$.
\end{exa}

\begin{figure}
   \includegraphics[scale=.6]{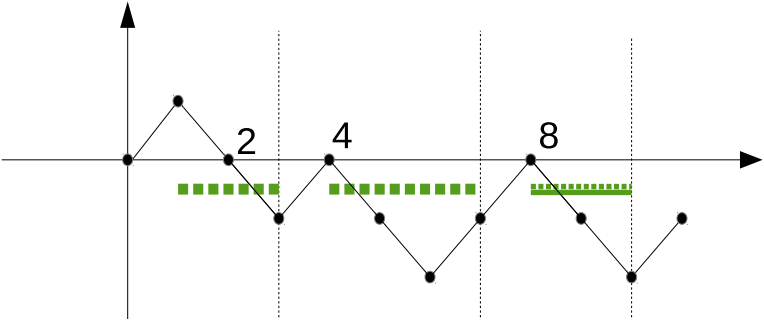}
\caption{A path in $\mathcal L(0010001000)$.}
\label{figurepath}
\end{figure}

The following result is a direct consequence of the definitions of the polynomials $\tilde \Omega_T$ and $\Upsilon_E$ and so we omit its proof.

\begin{prop}
\label{maincor}
 Let $n\in \PP$ and $T\in \two_s^{n-1}$. Then
 \[
  \tilde \Omega_T=\sum_{\Gamma\in \mathcal L(T)}(-1)^{\varepsilon(\Gamma)+\eta(\Gamma)+d_+(\Gamma)}q^{d_+(\Gamma)}.
 \]
\end{prop}

Our next target is to simplify the sum in Proposition \ref{maincor}. For this we introduce the following notation: if $T\in \two^{n-1}_s$, $s_0=0$ and $S(T)=\{s_1,\ldots,s_t\}_{<}$, we let $r_i\eqdef s_i+1$ for $i\in[0,t]$ and 
let 
\[\mathcal L_0(T)\eqdef\{\Gamma\in \mathcal L(T):\, \Gamma(r_i)=0 \textrm{ for some }i\in[t]\}.\] For example the path $\Gamma$ depicted in Figure \ref{figurepath} belongs to $\mathcal L_0(0010001000)$ as $\Gamma(4)=0$.

\begin{prop}
\label{l0}Let $n\in \PP$ and $T\in \two^{n-1}_s$. Then
 \[\sum_{\Gamma\in \mathcal L_0(T)}(-1)^{\varepsilon(\Gamma)+\eta(\Gamma)+d_+(\Gamma)}q^{d_+(\Gamma)}=0.\]
\end{prop}
\begin{proof}
 Let $\mathcal L^{(j)}_0(T)=\{\Gamma\in \mathcal L_0(T):\, \min\{i \in [t] :\, \Gamma(r_i)=0\}=j\}$.
 The result follows if we can find an involution
 \[
  \phi:\mathcal L^{(j)}_0(T)\rightarrow \mathcal L^{(j)}_0(T)
 \]
 such that 
 \begin{itemize}
  \item $d_+(\Gamma)=d_+(\phi(\Gamma))$,
  \item $\varepsilon(\Gamma)+\eta(\Gamma)\equiv \varepsilon(\phi(\Gamma))+\eta(\phi(\Gamma))+1 \pmod 2$,
\end{itemize}
for all $\Gamma\in \mathcal L^{(j)}_0(T)$.
The bijection $\phi$ is defined as follows. Fix an arbitrary path $\Gamma\in \mathcal L^{(j)}_0(T)$. Let $i$ be the maximum index smaller than $j$ such that $\Gamma(s_i)=0$. 
In the interval $[r_h,s_{h+1}]$, where $h\in [i, j-1]$, the path $\phi(\Gamma)$ is defined as follows (see Figure \ref{slalom2} for an illustration) 
\[
 \phi(\Gamma)(x)=\begin{cases}
\Gamma(x),&\begin{array}{l}\textrm{if there exist $a,b\in \mathbb N$ such that}\\ \textrm{$r_h<a<x<b<s_{h+1}$ and $\Gamma(a)=\Gamma(b)=0$,}\end{array}\\
-\Gamma(x),&\textrm{otherwise,}
              \end{cases}
\]
for all $x\in [r_h,s_{h+1}]$. 
\begin{figure}
   \includegraphics[scale=.6]{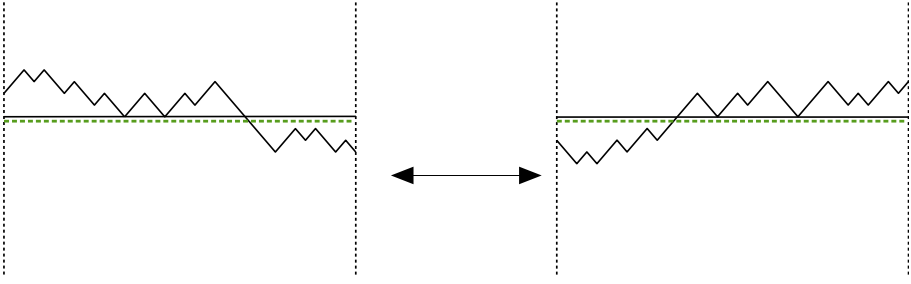}
\caption{The bijection $\phi$ in $[r_h,s_{h+1}]$.}
\label{slalom2}
\end{figure}
Finally we let $\phi(\Gamma)(x)=\Gamma(x)$ if $x\notin [r_i,s_j]$.
Note that, since $\Gamma(s_i)=\Gamma(s_j +1)=0$, we have that 
$|\Gamma(r_i)|=|\Gamma(s_j)|=1$ so $\phi(\Gamma)$ is still a lattice path.

Since $\Gamma(r_h)\neq 0$ and $\Gamma(s_{h+1})\neq 0$ for all $h\in [i,j-1]$ by construction, one can easily see that $\phi$ is an involution on $\mathcal L^{(j)}_0(T)$. It is also clear that $d_+(\Gamma)=d_+(\phi(\Gamma))$ as $\Gamma(n)=\phi(\Gamma)(n)$.

Now observe that $x_h(\Gamma) \equiv x_h(\phi(\Gamma))+1 \pmod 2$ for all $h\in [i+1,j]$ (see also Figure \ref{slalom2}), and that clearly $x_h(\Gamma)= x_h(\phi(\Gamma))$ if $h\notin [i+1,j]$. Therefore 
\begin{equation}\label{epsgamma} 
\varepsilon(\Gamma)+ \varepsilon(\phi(\Gamma))\equiv j-i \pmod 2.
\end{equation}

Finally, we have that (see Figure \ref{slalom2})
\begin{eqnarray}
\label{etagamma}
\nonumber
\eta (\Gamma )- \eta(\phi (\Gamma )) & = & \sum_{h=i}^{j-1}
(x_{h+1}(\phi (\Gamma ))-s_{h})-\sum_{h=i}^{j-1}(s_{h+1}-x_{h+1}(\Gamma ))\\
\nonumber
& \equiv & \sum_{h=i}^{j-1} (x_{h+1}(\phi (\Gamma ))-x_{h+1}(\Gamma )+s_{h+1}-s_{h}
) \\ \nonumber
& \equiv & j-i+s_{j}-s_{i} \\
& \equiv & j-i+1 \pmod{2}
\end{eqnarray}
since $\Gamma (s_{i})=\Gamma (s_{j}+1) =0$.

The result then follows from (\ref{epsgamma}) and (\ref{etagamma}).
\end{proof}
If $n$ is even other cancellations may occur in Proposition \ref{maincor} and to describe this, for $T\in \two_s^{n-1}$ we also let 
\[\mathcal L_0'(T)=\{\Gamma\in \mathcal L(T):\, \Gamma(x)=0 \textrm{ for some }x\in [r_{t},n]\}. 
\]
\begin{prop}\label{l0'}
   Let $n\in \PP$ be even and $T\in \two_s^{n-1}$. Then
\[
  \sum_{\Gamma\in \mathcal L'_0(T)\setminus \mathcal L_0(T)}(-1)^{\varepsilon(\Gamma)+\eta(\Gamma)+d_+(\Gamma)}q^{d_+(\Gamma)}=0. 
\]

\end{prop}
\begin{proof}
This proof is similar to that of Proposition \ref{l0}. We show that there exists an involution $\psi$ on $\mathcal L'_0(T)\setminus \mathcal L_0(T)$ such that $d_+(\Gamma)=d_+(\psi(\Gamma))$ and
  $\varepsilon(\Gamma)+\eta(\Gamma)\equiv \varepsilon(\psi(\Gamma))+\eta(\psi(\Gamma))+1 \pmod 2$,
for all $\Gamma\in \mathcal L'_0(T)\setminus \mathcal L_0(T)$.

If $\Gamma\in \mathcal L'_0(T)\setminus \mathcal L_0(T)$ let $j=\min\{x\in[r_t,n]:\, \Gamma(x)=0\}$ and $i=\max\{h\in[0,t]:\, \Gamma(s_h)=0\}$.  The path $\psi(\Gamma)$ is defined in the following way: if $r_i\leq x\leq s_t$ we let \[
 \psi(\Gamma)(x)=\begin{cases}
\Gamma(x),&\begin{array}{l}\textrm{if there exist $h\in[t]$ and $a,b\in \mathbb N$ such that}\\ \textrm{$r_{h-1}<a<x<b<s_{h}$ and $\Gamma(a)=\Gamma(b)=0$,}\end{array}\\
-\Gamma(x),&\textrm{otherwise.}
              \end{cases}
\]
Observe that in this situation $\Gamma(r_{h-1})\neq 0$ since $\Gamma\notin \mathcal L_0(T)$ and $\Gamma(s_h)\neq 0$ by the maximality of the index $i$.
Finally, we let $\psi(\Gamma)(x)=-\Gamma(x)$ if $x\in [r_t,j]$ and $\psi(\Gamma)(x)=\Gamma(x)$ if $x>j$ or $x\leq s_i$.

Note that, since $n-1$ is odd we have $\Gamma(x)\geq 0$ for all $x\geq j$. By reasoning as in the proof of Proposition \ref{l0} one obtains that $\varepsilon(\Gamma)+\varepsilon(\psi(\Gamma))\equiv t-i \pmod 2$ and $\eta(\Gamma)+\eta(\psi(\Gamma))\equiv t-i-1 \pmod 2$ (since $\Gamma(s_i)=\Gamma(j)=0$), thereby completing the proof.
\end{proof}

Propositions \ref{l0} and \ref{l0'} lead us to consider the set of paths $\tilde{\mathcal L }(T)$ given by $\tilde {\mathcal L}(T)\eqdef \mathcal L(T)\setminus (\mathcal L_0(T)\cup \mathcal L_0'(T))$ if $n$ is even and $\tilde {\mathcal L}(T)\eqdef \mathcal L(T)\setminus \mathcal L_0(T)$ if $n$ is odd associated to $T\in \two^{n-1}_s$. 
Note that $\Gamma(n) \neq 0$ if $\Gamma \in \tilde{\mathcal L }(T)$.
Figure \ref{slalompath} shows an example of a path in $\tilde{\mathcal L}(T)$.

\begin{figure}
   \includegraphics[scale=.5]{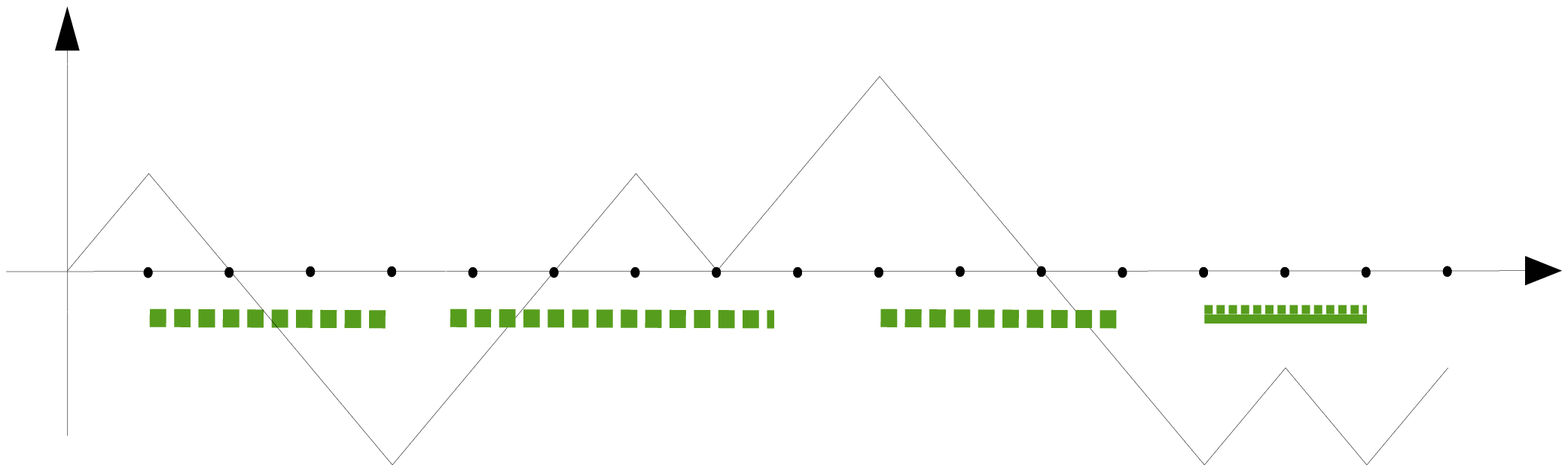}
\caption{A path in  $\tilde {\mathcal L}(0001000010001000)$.}
\label{slalompath}
\end{figure}

Now we want to show that there are no further cancellations in the sum appearing in Proposition \ref{maincor}.  More precisely, if we let $\chi_{\Gamma(n)>0}=1$ if $\Gamma(n)>0$ and $\chi_{\Gamma(n)>0}=0$ otherwise, we have the following result.
\begin{thm}\label{ltilde}
 Let $n\in \PP$, $T\in \two^{n-1}_s$ and $\Gamma\in \tilde{\mathcal {L}}(T)$. Then
 \[
  \varepsilon(\Gamma)+\eta(\Gamma)\equiv \chi_{\Gamma(n)>0}(n-1)+\sum_{h=1}^{t} {r_h}\pmod 2.
 \]
\end{thm}
\begin{proof}
 Suppose $i,k\in [0,t]$, $i<k$ are such that $\Gamma(s_i)=\Gamma(s_k)=0$ and $\Gamma(s_h)\neq 0$ for all $h\in (i,k)$.
 We consider $\eta_{i,k}(\Gamma)\eqdef|\{x\in [r_{i},s_k]:\, \Gamma(x)\geq 0\}|$ and $\zeta_{i,k}(\Gamma)\eqdef\sum_{h={i+1}}^{k}x_h(\Gamma)$ and we claim that
 \begin{equation}
  \eta_{i,k}(\Gamma)+\zeta_{i,k}(\Gamma)\equiv 0 \pmod 2.
 \end{equation}
 If $k-i$ is odd the number of maximal intervals contained in $[r_i,s_k]$ where $\Gamma$ takes nonnegative values is $\frac{k-i+1}{2}$ and all such intervals contain an odd number of elements. 
If $k-i$ is even  there are $\frac{k-i}{2}$ such intervals with an odd number of elements and one interval with an even number of elements. In both cases we deduce that $\eta_{i,k}(\Gamma)\equiv  \lfloor \frac{k-i+1}{2}\rfloor$.
 Now we observe that $x_{k}(\Gamma)\equiv x_{k-2}(\Gamma)\equiv \cdots \equiv 1 \pmod 2$ and $x_{k-1}(\Gamma)\equiv x_{k-3}(\Gamma)\equiv \cdots \equiv 0 \pmod 2$. Therefore
 \[
 \zeta_{i,k}(\Gamma)= \sum_{h=i+1}^{k}x_h(\Gamma)\equiv \lfloor \frac{k-i+1}{2}\rfloor  \pmod 2
 \]
 and so
 \begin{align*}
  \eta_{i,k}(\Gamma)+\zeta_{i,k}(\Gamma)&\equiv \lfloor \frac{k-i+1}{2}\rfloor+\lfloor \frac{k-i+1}{2}\rfloor \equiv 0 \pmod 2
 \end{align*}
 Now let $i$ be the maximum index such that $\Gamma(s_i)=0$. We let in this case $\eta_{i,t+1}(\Gamma)\eqdef|\{x\in [r_{i},n-1]:\, \Gamma(x)\geq 0\}|$ and $\zeta_{i,t+1}(\Gamma)\eqdef\sum_{h=i+1}^{t}x_h(\Gamma)$. We leave to the reader to verify that if $\Gamma(n)>0$ and $n$ is even then $\eta_{i,t+1}(\Gamma)\equiv \lfloor \frac{t-i-1}{2}\rfloor$ and $\zeta_{i,t+1}(\Gamma) \equiv \lfloor \frac{t-i+1}{2}\rfloor$ and therefore
\[
   \eta_{i,t+1}(\Gamma)+\zeta_{i,t+1}(\Gamma)\equiv 1.
\]
In all the other cases we have $\eta_{i,t+1}(\Gamma)+\zeta_{i,t+1}(\Gamma)\equiv 0$; in fact, with an argument similar to the one used in the previous case one can show that:
\begin{itemize}
   \item if $n$ is even and $\Gamma(n)<0$ we have $\eta_{i,t+1}(\Gamma)\equiv\zeta_{i,t+1}(\Gamma)\equiv \lfloor \frac{t-i}{2}\rfloor$;
\item if $n$ is odd and $\Gamma(n)>0$ we have $\eta_{i,t+1}(\Gamma)\equiv \zeta_{i,t+1}(\Gamma)\equiv\lfloor \frac{t-i+1}{2}\rfloor$;
\item if $n$ is odd, $\Gamma(n)<0$ and $\Gamma(r_t)<0$ we have $\eta_{i,t+1}(\Gamma)\equiv \zeta_{i,t+1}(\Gamma)\equiv\lfloor \frac{t-i}{2}\rfloor$;
\item if $n$ is odd, $\Gamma(n)<0$ and $\Gamma(r_t)>0$ we have $\eta_{i,t+1}(\Gamma)\equiv \zeta_{i,t+1}(\Gamma)\equiv\lfloor \frac{t-i-1}{2}\rfloor$.
\end{itemize}

 Now we can conclude the proof. 
Let $\Gamma\in \tilde{\mathcal {L}}(T)$ and let $\{i_1,\ldots,i_z\}_<=\{h\in [0,t]:\, \Gamma(s_h)=0\}\cup\{t+1\}$. 
We have
 \begin{align*}
  \eta(\Gamma)+\varepsilon(\Gamma)&\equiv \sum_{v=1}^{z-1} \left( \eta_{i_v,i_{v+1}}(\Gamma)+\zeta_{i_v,i_{v+1}}(\Gamma) \right) +r_1+\cdots+r_t \pmod 2\\
  &\equiv \chi_{\Gamma(n)>0}(n-1)+r_1+\cdots+r_{t}\pmod 2.
\end{align*}

\end{proof}
\begin{cor}
\label{combinter}
Let $n\in \PP$ and $T\in \two^{n-1}_s$. Then
\[
 [q^i]\tilde \Omega_T=(-1)^{i+r_1+\cdots+r_{t}+(n-1) \chi_{2i>n}}|\{\Gamma\in \tilde{\mathcal {L}}(T):\Gamma(n)=2i-n\}|,
\]
where $\chi_{2i>n}=1$ if $2i>n$ and $\chi_{2i>n}=0$ otherwise.
\end{cor}
\begin{proof}
 This follows immediately from Propositions \ref{maincor}, \ref{l0} and \ref{l0'}, and Theorem \ref{ltilde}, together with (\ref{dplus}).
\end{proof}

\begin{cor}\label{skew}
Let $n\in \PP$, and $T\in \two^{n-1}_s$. Then
$$
q^n \tilde \Omega_T \left( \frac{1}{q} \right) = - \tilde \Omega_T(q).
$$
\end{cor}
\begin{proof}
   A bijection $\phi$ constructed as in the proof of Proposition \ref{l0} shows that 

\[|\{\Gamma\in \tilde {\mathcal L}(T):\Gamma(n)=n-2i\}|=|\{\Gamma\in \tilde {\mathcal L}(T):\Gamma(n)=2i-n\}|.
\]
and the result follows from Corollary \ref{combinter}.
\end{proof}

Let $T\in \two^{n-1}_s$. Note that a lattice path $\Gamma$ is a $T$-slalom if and only if $\Gamma\in\tilde {\mathcal L}(T)$ and $\Gamma(n)>0$. We can now prove Theorem \ref{finalmain}.

{\em Proof of Theorem \ref{finalmain}.}
Corollary \ref{skew} shows that there exists an involution $\phi$ on $\tilde{\mathcal L}(T)$ such that $d_+(\Gamma)=d_-(\phi(\Gamma))$. The result then follows immediately from Proposition \ref{KLOmega} and Corollary \ref{combinter} together with the well-known fact that $\deg P_{u,v}\leq \lfloor\frac{\ell-1}{2}\rfloor$.
$\Box$

We illustrate the preceding theorem with some examples. If $\ell(u,v)=1$
then we have from Theorem \ref{finalmain} and our definitions that 
\[
P_{u,v}(q) 
= q^{-\frac{\ell(\varepsilon)}{2}} \, b_{\varepsilon} \, \Omega_{\varepsilon}(q)
= 1.
\] 
Similarly we obtain
\[
P_{u,v}(q)
= b_0 \, \Omega_0(q) 
= 1 
\]
if $\ell(u,v)=2$ (where we have used the fact that $b_0(u,v)=1$ if $\ell(u,v)=2$) 
and
\begin{align*}
P_{u,v}(q) 
& = b_{00} \, \Omega_{00}(q)+ b_{01} \, \Omega_{01}(q) 
+ q \, b_{\varepsilon} \, \Omega_{\varepsilon}(q) \\
& = b_{00}(1-2q) + b_{01}(q)+q \, b_{\varepsilon} \\
& = 1+q(-2+b_{01}+ b_{\varepsilon}), \\
\end{align*}
if $\ell(u,v)=3$.

We feel that the formula obtained in Theorem \ref{finalmain} is the simplest and most explicit nonrecursive
combinatorial formula known for the Kazhdan-Lusztig polynomials that holds in complete generality since this formula, as the one in \cite[Corollary 3.2]{BB}, expresses the Kazhdan-Lusztig polynomial
of $u,v \in W$ as a sum of at most $f_{\ell(u,v)}$ summands, as opposed to 
$2^{\ell(u,v)}+2^{\ell(u,v)-2}+\cdots$ for the one obtained in \cite[Theorem 7.2]{BreJAMS}, 
each one of which is the product of a number, which depends on $u$, $v$, and $W$, with a
polynomial, that is independent of $u$, $v$, and $W$. However, this formula is more explicit than the
one obtained in \cite[Corollary 3.2]{BB} since in the formula obtained in \cite{BB} the 
polynomials have a combinatorial interpretation, but no combinatorial interpretation is known for 
the numbers, while in the formula obtained in Theorem \ref{finalmain} both the numbers and the
polynomials have a combinatorial interpretation.

\section{Linear relations for Bruhat paths}

The main goal of this section is to study linear relations satisfied by the functions $b(u,v)$ for all $u,v \in W$ and all Coxeter groups $W$.
 As a consequence of our results we show that the formula appearing in Theorem \ref{finalmain} cannot be linearly simplified.

\subsection{Construction of reflection orderings}

We start with a general construction of reflection orderings. Let $(W,S)$ be a Coxeter system, $\Pi$ be the associated set of simple roots, and $\Phi^+=\Phi^+(W)$ the associated set of positive roots. A \emph{weight} function on $\Phi^+$ is a map $p:\Phi^+\rightarrow \mathbb R_{\geq 0}$ which is linear, in the sense that if $\beta=c_1\beta_1+c_2 \beta_2$, with $\beta,\beta_1,\beta_2\in \Phi^+$ and $c_1,c_2\in \mathbb N$, then $p(\beta)=c_1p(\beta_1)+c_2p(\beta_2)$.
It is clear that a weight function $p$ is uniquely determined by its images on $\Pi$ and that the set $\Phi_0^+(p)=\{\beta\in \Phi^+:\, p(\beta)=0\}$ is the set of positive roots of a parabolic subgroup of $W$.
Let $I=(\alpha_1,\ldots,\alpha_l)$ be an indexing (total ordering) of the elements in $\Pi$. 
Then the associated lexicographic order on the root space $\mathbb R\alpha_1\oplus\cdots\oplus \mathbb R\alpha_l$ is given by $\sum {c_i\alpha_i}<\sum d_i\alpha_i$ if $(c_1,\ldots,c_l)$ is smaller than $(d_1,\ldots,d_l)$ lexicographically.

   Let $W$ be a Coxeter group, $p$ be a weight function on $\Phi^+(W)$, and $W'$ the parabolic subgroup of $W$ given by $\Phi^+(W')=\Phi_0^+(p)$. Let $\prec$ be a reflection ordering on $\Phi^+(W')$ and $I$ an indexing of $\Pi$. Then we define a total ordering $\ll$ on $\Phi^+$ depending on $p,\prec, I$ in the following way: for $\beta,\beta'\in \Phi^+$ we let $\beta\ll \beta'$ if one of the following conditions apply:
\begin{itemize}
   \item $p(\beta)=p(\beta')=0$ and $\beta\prec \beta'$;
\item $p(\beta)\neq 0$ and $p(\beta')=0$;
\item $p(\beta),p(\beta')\neq 0$ and $\frac{\beta}{p(\beta)}<\frac{\beta'}{p(\beta')}$ in the lexicographic order associated to $I$.
\end{itemize}
It is clear that $\ll$ is a total ordering on $\Phi^+(W)$.
\begin{prop}
\label{ll}
The total ordering $\ll$ on $\Phi^+(W)$ constructed above is a reflection ordering. 
\end{prop}
\begin{proof}
   We have to show that if $\beta=c_1\beta_1+c_2 \beta_2$, with $\beta,\beta_1,\beta_2\in \Phi^+$, $c_1,c_2\in \mathbb R_{>0}$, and $\beta_1\ll \beta_2$ then
$\beta_1\ll \beta\ll\beta_2$. 
\begin{itemize}
   \item If $p(\beta_1)=p(\beta_2)=0$ then $\beta_1,\beta_2\in  \Phi^+(W')$ and hence also $\beta\in \Phi^+(W')$; the result follows since $\prec$ is a reflection ordering on $ \Phi^+(W')$;
\item if $p(\beta_1)\neq 0$ and $p(\beta_2)=0$ then $p(\beta)=c_1p(\beta_1)>0$ and in particular $\beta\ll\beta_2$. 
Moreover, if we denote by $x_i(\beta)$ the $i$-th coordinate of $\beta$ with respect to the chosen indexing $I$ on $\Pi$ 
(so $\beta = \sum_{i=1}^{l} x_i(\beta) \alpha_i$) we have
\[
   \frac{x_i(\beta)}{p(\beta)}=\frac{x_i(c_1\beta_1+c_2\beta_2)}{c_1p(\beta_1)}=\frac{c_1x_i(\beta_1)+c_2x_i(\beta_2)}{c_1p(\beta_1)}\geq \frac{x_i(\beta_1)}{p(\beta_1)}.
\]
\item if $p(\beta_1),p(\beta_2)\neq 0$ then
\[
  \frac{x_i(\beta)}{p(\beta)}= \frac{x_i(c_1\beta_1+c_2\beta_2)}{c_1p(\beta_1)+c_2p(\beta_2)}=\frac{c_1p(\beta_1)\frac{x_i(\beta_1)}{p(\beta_1)}+c_2p(\beta_2)\frac{x_i(\beta_2)}{p(\beta_2)}}{c_1p(\beta_1)+c_2p(\beta_2)} 
\]
which shows that $\frac{x_i(\beta)}{p(\beta)}$ is a convex linear combination of $\frac{x_i(\beta_1)}{p(\beta_1)}$ and $\frac{x_i(\beta_2)}{p(\beta_2)}$, 
completing the proof.
\end{itemize}
\end{proof}

Note that Proposition \ref{ll} vastly generalizes Proposition 5.2.1 of 
\cite{BjBr}.

\begin{cor}\label{goodorder}
   Let $(W,S)$ be a Coxeter system and $P$ be the maximal parabolic subgroup generated by $S\setminus\{s\}$, for some $s\in S$. Then there exists a reflection ordering $\ll$ on $ \Phi^+$ such that
\begin{itemize}                                                                                                                                                                 
\item  $t\ll s$ for every reflection $t$ in $W$;
\item if $t$ is a reflection in $P$ then $t\ll sts$; 
\item if $t$ and $t'$ are reflections in $P$ then $t\ll t'$ if and only if $sts\ll st's$.                                                                                                                                                   \end{itemize}
\end{cor}
\begin{proof}
   Consider an indexing $I$ of $\Pi=\{\alpha_1,\ldots,\alpha_l\}$ with $\alpha_l=\alpha_s$ and the weight function given by $p(\alpha_i)=1$ if $i<l$ and $p(\alpha_l)=0$. Let $\ll$ be the reflection ordering constructed above with respect to $p$ and $I$ (there is no choice for $\prec$ in this case).
It is clear that $\alpha_s$ is the maximal element. If $t$ is a reflection in $P$ we have that
\[
   \alpha_{sts}=s(\alpha_t)=\alpha_t+c\alpha_l,
\]
for some nonnegative integer $c$. In particular $p(\alpha_{sts})=p(\alpha_t)$ and,
by Proposition \ref{ll}, $\alpha_t\ll \alpha_{sts} \ll \alpha_s$.
Now let $t,t'$ be reflections in $P$. We clearly have $p(t),p(t')\neq 0$. Since $p(\alpha_{sts})=p(\alpha_t)$ and  all the coordinates but the last one of $\alpha_t$ and $\alpha_{sts}$ coincide (and similarly for $t'$) we deduce that $t\ll t'$ if and only if $sts\ll st's$.

\end{proof}

\subsection{The pyramid over a Bruhat interval}
Let $(W,S)$ be a Coxeter system and $[u,v]$ be an interval in $W$. We say that an interval $[u,vs]$ is a \emph{pyramid} over $[u,v]$ if $s\in S$ and $s\not\leq v$. The name pyramid comes from the fact that if $[u,v]$ is isomorphic as a poset to the face lattice of a polytope $P$ then $[u,vs]$ is isomorphic to the face lattice of a pyramid over $P$.

The following result states that the complete $cd$-index of a pyramid over a Bruhat interval does not depend on $s$, generalizes \cite{ER} and expresses the complete $cd$-index of the pyramid $[u,vs]$ in terms of the complete $cd$-index of $[u,v]$ and of smaller intervals.

\begin{prop}\label{pyr}
   Let $[u,v]$ be a Bruhat interval and $[u,vs]$ be a pyramid over $[u,v]$. Then
\[
   \tilde \Psi_{u,vs}=\frac{1}{2}\Big(\tilde \Psi_{u,v}c+c\tilde \Psi_{u,v}+\sum_{x\in(u,v)}\tilde \Psi_{u,x}d\tilde \Psi_{x,v}\Big).
\]
In particular $\tilde \Psi_{u,vs}$ does not depend on $s$.
\end{prop}
\begin{proof}
We start with an observation. If $x<v$ then any path $\Delta$ in the Bruhat graph $\Delta=(x_0\stackrel{t_1}{\longrightarrow} x_1\stackrel{t_2}{\longrightarrow}\cdots \stackrel{t_{r+1}}{\longrightarrow}x_{r+1})$ from $x$ to $v$ corresponds to a path 
$\Delta'=(x_0s\stackrel{st_1s}{\longrightarrow} x_1s\stackrel{st_2s}{\longrightarrow}\cdots \stackrel{st_{r+1}s}{\longrightarrow}x_{r+1}s)$ from $xs$ to $vs$. This correspondence is a bijection between paths from $x$ to $v$ and paths from $xs$ to $vs$; 
moreover, if we consider the reflection ordering $\ll$ defined in Corollary \ref{goodorder}, we have that if $\Delta$ corresponds to $\Delta'$ in this correspondence then $m_\ll(\Delta)=m_\ll(\Delta')$.
We also observe that if we consider the \emph{lower $s$-conjugate} $\ll_s$ of $\ll$ (see \cite[Proposition 5.2.3]{BB}) given in this case by $r\ll_sr'$ if and only if either $r=s$ or $srs\ll sr's$, we still obtain $m_{\ll_s}(\Delta)=m_{\ll_s}(\Delta')$.

   Given a path $\Delta=(u_0\stackrel{t_1}{\longrightarrow} u_1\stackrel{t_2}{\longrightarrow}\cdots \stackrel{t_{r+1}}{\longrightarrow}u_{r+1})$ from $u$ to $vs$ there exists a unique $x\in[u,v]$ such that $u_i=x$ and $u_{i+1}=x s$. 
We denote it by $x(\Delta)$ and in the computation of the complete $cd$-index
\[
   \tilde \Psi_{u,vs}=\sum_{\Delta\in B(u,vs)}m(\Delta)
\]
we split the sum on the right-hand side according to $x(\Delta)$. We first consider the reflection ordering $\ll$. In this case we have
\[
   \sum_{\{ \Delta\in B(u,vs):\,x(\Delta)=u \}}m_\ll(\Delta)=b\cdot \sum_{\Delta'\in B(us,vs)}m_\ll(\Delta')=b\cdot \sum_{\Delta\in B(u,v)}m_\ll(\Delta)=b\cdot \tilde \Psi_{u,v}.
\]
and
\[
 \sum_{\{ \Delta\in B(u,vs):\,x(\Delta)=v \}}m_\ll(\Delta)=\sum_{\Delta\in B(u,v)}m_\ll(\Delta)\cdot a=\tilde \Psi_{u,v}\cdot a,
\]
where we have used the fact that $s$ is the maximal reflection in the ordering $\ll$ and the observation at the beginning of the present proof.
If $x \in (u,v)$ we have
\[
\sum_{\{ \Delta\in B(u,vs):\,x(\Delta)=x \}}m_\ll(\Delta)=\sum_{\Delta\in B(u,x)}m_\ll(\Delta)\cdot ab\cdot \sum_{\Delta'\in B(xs,vs)}m_\ll(\Delta')=\tilde \Psi_{u,x}ab\tilde \Psi_{x,v}
\]
and we conclude that
\[
\tilde \Psi_{u,vs}=b\cdot \tilde \Psi_{u,v}+\tilde \Psi_{u,v}\cdot a+\sum_{x\in(u,v)}\tilde \Psi_{u,x}ab\tilde \Psi_{x,v}.
\]
By reasoning in a similar way with the ordering $\ll_s$ we can obtain the analogous formula
\[
\tilde \Psi_{u,vs}=a\cdot \tilde \Psi_{u,v}+\tilde \Psi_{u,v}\cdot b+\sum_{x\in(u,v)}\tilde \Psi_{u,x}ba\tilde \Psi_{x,v}
\]
and the result follows by ``averaging'' these two expressions for $\tilde \Psi_{u,vs}$.
\end{proof}

Consider the derivation $D_d$ on $\mathcal A$. One easily checks that $D_d$ restricts to a derivation on the space of $cd$-polynomials as
$\delta(c)=\delta(a+b)=2 (1\otimes 1)$ and so $D_d(c)=2d$ and $\delta(d)=\delta(ab+ba)=a\otimes 1+1\otimes b +b\otimes 1+1\otimes a$ and so $D_d(d)=ad+db+bd+da=dc+cd$. Corollary \ref{der} and Proposition \ref{pyr} therefore allow us to write
\[
   \tilde \Psi_{u,vs}=\frac{1}{2}\Big(\tilde \Psi_{u,v}c+c\tilde \Psi_{u,v}+D_d(\tilde \Psi_{u,v})\Big)
\]
which shows that $\tilde \Psi_{u,vs}$ depends on $\tilde \Psi_{u,v}$ only. 
Let $G'$ be the derivation on $\mathcal A$ given by $G'(a)=ab$ and $G'(b)=ba$ so that $G'(c)=d$ and $G'(d)=dc$. 
The next result is then a consequence of \cite[Lemma 5.1 and Theorem 5.2]{ER}. 

\begin{cor}\label{finalpyr}
Let $(W,S)$ be a Coxeter system, $u,v \in W$, $u < v$, and $s\in S$ 
be such that $s\not\leq v$. Then
\end{cor}
\[
   \tilde \Psi_{u,vs}=c\tilde \Psi_{u,v}+G'(\tilde \Psi_{u,v}).
\]

Similarly, one can prove the following ``left version'' of Corollary 
\ref{finalpyr}.
\begin{cor}\label{finalpyrl}
Let $(W,S)$ be a Coxeter system, $u,v \in W$, $u < v$, and $s\in S$ 
be such that $s\not\leq v$. Then
\end{cor}
\[
   \tilde \Psi_{u,sv}=c\tilde \Psi_{u,v}+G'(\tilde \Psi_{u,v}).
\]

\subsection{3-complete Coxeter systems}
Let $(W,S)$ be the Coxeter system of rank $l$ such that $m(s,s')=3$ for all $s,s'\in S$, $s\neq s'$. 
We call this the 3-complete Coxeter system (or group) of rank $l$.

Our first result can be interpreted as a concrete criterion to determine the set of (left) descents of a generic element in a 3-complete Coxeter group: it is used in the sequel in the construction of reflection orderings, but is interesting in its own right.

Let $(W,S)$ be a 3-complete Coxeter system of rank $l$, $S=\{s_1,\ldots,s_l\}$ and let $W'$ be the parabolic subgroup of $W$ generated by $S\setminus\{s_1\}$ (note that $W'$ is a 3-complete Coxeter group of rank $l-1$). 
We also let $\Pi=\{\alpha_1,\ldots,\alpha_l\}$ where $\alpha_i$ is the simple root corresponding to $s_i$ for all $i\in[l]$.
We observe that
\[
   s_i(\alpha_j)=\begin{cases}
                    \alpha_i+\alpha_j&\textrm{if $i\neq j$;}\\ 
                    -\alpha_i &\textrm{if $i=j$.}
                 \end{cases}
\]
We consider the $W'$-orbit of the simple root $\alpha_1$ to study the sets of left descents of elements in $W'$. 
We adopt the following notation: for all $w\in W'$ we let $c_i(w), d_i(w)\in \mathbb Z$, $i\in [l]$, be given by 
\[w(\alpha_1)=\sum_{i=1}^{l}c_i(w)\alpha_i
\]
and $d_i(w)\eqdef 2c_i(w)-\sum_{k\neq i}c_k(w)$.
Before proving the main result about the coefficients $d_i(w)$ we need the following preliminary result.
\begin{lem}\label{recursiondi}
   Let $u\in W'$. Then for all $i,j\in [2,l]$, $i\neq j$, we have
\begin{enumerate}
\item[(a)] $d_i(s_iu)=-d_i(u)$;
\item[(b)] $d_j(s_iu)=d_i(u)+d_j(u)$.
\end{enumerate}
\end{lem}
\begin{proof}
  We first observe that for all $u\in W'$ we have
\begin{equation}\label{ciu}
   c_i(s_ju)=\begin{cases}c_i(u),& \textrm{if $i\neq j$;}\\ \sum_{k\neq i}c_k(u)-c_i(u), &  \textrm{if $i=j$.}   \end{cases}
\end{equation}
 (a) We have
\[
   d_i(s_iu)=2c_i(s_iu)-\sum_{k\neq i}c_k(s_iu)= 2\sum_{k\neq i}c_k(u)-2c_i(u)-\sum_{k\neq i}c_k(u)=-d_i(u).
\]

(b) We have
\begin{align*}
   d_j(s_iu)&=2c_j(s_iu)-\sum_{k\neq i,j}c_k(s_iu)-c_i(s_iu)=2c_j(u)-\sum_{k\neq i,j}c_k(u)-\sum_{k\neq i}c_k(u)+c_i(u)\\
&=2c_j(u)-\sum_{k\neq j}c_k(u)+2c_i(u)-\sum_{k\neq i}c_k(u)=d_i(u)+d_j(u).
\end{align*}
\end{proof}
For $w\in W$ we let $\Des_L(w)\eqdef\{i\in [l]:\, s_iw<w\}$.
\begin{prop}\label{dides}
   Let $w\in W'$ and $i\in [2,l]$. Then $d_i(w)\neq 0$ and
\[
  d_i(w)>0\Leftrightarrow  i\in \Des_L(w).
\]
\end{prop}
\begin{proof}
   We proceed by induction on $\ell(w)$. If $\ell(w)=0$ then $d_i(w)=-1$ and $i\notin \Des_L(w)$, and the statement is true.
So let $\ell(w)\geq 1$. 

If $i\in \Des_L(w)$ let $w=s_iu$, with $i\notin \Des_L(u)$. By the induction
 hypothesis we have $d_i(u)<0$ and so, by Lemma \ref{recursiondi}, we have $d_i(w)=-d_i(u)>0$.

If $i\notin \Des_L(w)$ let $j$ be such that $j\in \Des_L(w)$ and $w=s_ju$, with $j\notin \Des_L(u)$. 
Now two cases occur:
 if $i\notin \Des_L(u)$ we have by induction that $d_i(u),d_j(u)<0$ and so, by Lemma \ref{recursiondi} we conclude that
$d_i(w)=d_i(u)+d_j(u)<0$. If $i\in \Des_L(u)$ we let $\tilde u$ be such that $w=s_js_i\tilde u$, with $\ell(w)=\ell(\tilde u)+2$. We have that $j\notin \Des_L(\tilde u)$ since otherwise we could obtain a reduced expression for $w$ starting with $s_i$. Therefore, by induction we have $d_j(\tilde u)<0$.  
Using Lemma \ref{recursiondi} we then conclude that $d_i(w)=d_i(s_i\tilde u)+d_j(s_i\tilde u)=-d_i(\tilde u)+d_j(\tilde u)+d_i(\tilde u)=d_j(\tilde u)<0$.
\end{proof}
\begin{cor}\label{ht>l}
   For all $w\in W'$ we have $\mathrm{ht}(w(\alpha_1))\geq \ell(w)+1$, where  $\mathrm{ht}$ denotes the height function defined by $\mathrm{ht}(\sum c_i\alpha_i)=\sum c_i$.
\end{cor}
\begin{proof}
   We proceed by induction on $\ell(w)$, the result being trivial if $\ell(w)=0$. So let $\ell(w)>0$, $i\in \Des_L(w)$ and $w=s_iu$. Using Eq. \eqref{ciu} we easily have that $c_i(w)=c_i(u)-d_i(u)$. Therefore we have 
\[
   \mathrm{ht}(w(\alpha_1))=\mathrm{ht}(u(\alpha_1))-d_i(u)\geq \ell(u)+1-d_i(u)=\ell(w)-d_i(u)
\]
and the result follows since $d_i(u)<0$ by Proposition \ref{dides}.

\end{proof}

We now show the existence of reflection orderings in a 3-complete Coxeter group satisfying some particular properties.
 
\begin{lem}\label{szszs}Let $(W,S)$ be a 3-complete Coxeter system, $s\in S$ and $P$ be the parabolic subgroup of $W$ generated by $S\setminus\{s\}$. Then there exists a reflection ordering $\ll$ such that for any reflection $t\in P$ and any element $z\in P$, $\ell(z)\geq 2$, we have
\[
t\ll szsz^{-1}s\ll sts\ll s.
\] 
\end{lem}
\begin{proof}
We consider the reflection ordering $\ll$ constructed as in Proposition \ref{ll}, where the weight $p=\mathrm{ht}$ is the height function, and the indexing $I=(\alpha_1,\ldots,\alpha_l)$ is such that $\alpha_1=\alpha_s$ (there is no choice for $\prec$ here, since $\Phi_0^+(\mathrm{ht})=\emptyset$), so $s\in S$ is the simple reflection corresponding to $\alpha_1$. 
It is enough to show that
\[
\frac{x_1(\alpha_t)}{\mathrm{ht}(\alpha_t)}<\frac{x_1(sz(\alpha_s))}{\mathrm{ht}(sz(\alpha_s))}<\frac{x_1(s(\alpha_t))}{\mathrm{ht}(s(\alpha_t))}<\frac{x_1(\alpha_s)}{\mathrm{ht}(\alpha_s)}.
\]
Since $x_1(\alpha_t)=0$ and $\frac{x_1(\alpha_s)}{\mathrm{ht}(\alpha_s)}=1$ we have to show that
\[
0<\frac{x_1(sz(\alpha_s))}{\mathrm{ht}(sz(\alpha_s))}<\frac{x_1(s(\alpha_t))}{\mathrm{ht}(s(\alpha_t))}<1.
\]
Recall that we have $r(\alpha_{r'})=\alpha_{r}+\alpha_{r'}$ for all $r,r'\in S$, $r\neq r'$. In particular we have, since $t \in P$
\[
s(\alpha_t)=\alpha_t+\mathrm{ht}(\alpha_t)\alpha_s.
\]
It follows that $\frac{x_1(s(\alpha_t))}{\mathrm{ht}(s(\alpha_t))}=\frac{1}{2}$ and so to conclude the proof we only have to show that
\[
0<\frac{x_1(sz(\alpha_s))}{\mathrm{ht}(sz(\alpha_s))}<\frac{1}{2}
\]
for all $z\in P$, $\ell(z)\geq 2$. So let $z(\alpha_s)=\alpha_s+\sum_{i\geq 2}c_i\alpha_i$. By Corollary \ref{ht>l} we have $\mathrm{ht}(z(\alpha_s))=1+\sum_{i\geq 2} c_i\geq \ell(z)+1$ and in particular we have $c\eqdef\sum_{i\geq 2} {c_i}\geq 2$. 
Therefore $sz(\alpha_s)=(c-1)\alpha_s+\sum_{i\geq 2}c_i\alpha_i$ and so $x_1(sz(\alpha_s))=c-1$ and $\mathrm{ht}(sz(\alpha_s))=2c-1$. 
The result follows.

\end{proof}

\begin{prop}\label{szrzs}
Let $(W,S)$ be a 3-complete Coxeter system. Let $r,s\in S$, $r\neq s$. Let $P$ be the parabolic subgroup of $W$ generated by $S\setminus\{r,s\}$. Then for every $t,z,w\in P$, $t$ a reflection, $\ell(z)\geq 2$, we have 
\[
t\ll szsz^{-1}s\ll sts\ll s
\]
and
\[
   t\ll swrw^{-1}s\ll sts.
\]
Moreover $t\ll wsw^{-1}$, $t\ll wrw^{-1}$ and $r\ll s$.
\end{prop}
\begin{proof}
We consider the weight $p$ given by $p(\alpha)=1$ for all $\alpha\in \Pi\setminus\{\alpha_r\}$ and $p(\alpha_r)=2$. 
We also consider an indexing $I$ of $\Pi=\{\alpha_1,\alpha_2,\ldots,\alpha_\ell\}$ such that $\alpha_1=\alpha_s$ and $\alpha_2=\alpha_r$ and we let $\ll$ be the reflection ordering associated to $p$ and $I$ (again there is no choice for $\prec$ as $\Phi^+_0(p)=\emptyset$).
   As the restriction of $\ll$ to the parabolic subgroup generated by $S\setminus\{r\}$ is the reflection ordering considered in Lemma \ref{szszs} the first part of the statement follows.

 For the second part we first observe that $w(\alpha_r)=\alpha_r+\sum_{i\geq 3} c_i\alpha_i$. Let $c\eqdef \sum_{i \geq 3}c_i\geq 0$. Then $sw(\alpha_r)=(c+1)\alpha_s+\alpha_r+\sum_{i\geq 3} c_i\alpha_i$ and $x_1(sw(\alpha_r))=c+1\geq 1$, which implies $t\ll swrw^{-1}s$. Moreover we have
\[
   \frac{x_1(sw(\alpha_r))}{p(sw(\alpha_r))}=\frac{c+1}{2c+3}<\frac{1}{2}=\frac{x_1(s(\alpha_t))}{p(s(\alpha_t))}
\]
implying $swrw^{-1}s\ll sts$. The relations $t\ll wsw^{-1}$, $t\ll wrw^{-1}$ and $r\ll s$ are all clear from the definition.
\end{proof}

We can now prove the second main result of this section.

\begin{thm}\label{finalsvs}Let $(W,S)$ be a 3-complete Coxeter system. Let $r,s\in S$, $r\neq s$, and $P$ be the parabolic subgroup generated by $S\setminus\{s,r\}$. Then for all $v\in P$, $v\neq e$,  we have
\[
   \tilde \Psi_{e,svs}+d\cdot \tilde \Psi_{e,v}=\tilde \Psi_{e,rvs}+\tilde \Psi_{e,v}.
\]
\end{thm}
\begin{proof}
   We establish the result by means of an explicit bijection. In particular we exhibit a bijection $\sigma$ between $B(e,svs)\cup B(e,v)\cup \overline {B(e,v)}$ and $B(e,rvs)\cup B(e,v)$, where $\overline {B(e,v)}$ is just a copy of $B(e,v)$, which is well-behaved with respect to the contributions of these paths to the corresponding complete $cd$-indices in the following sense. 
If $\Delta\in B(e,svs)$ or $\Delta\in B(e,rvs)$ we consider the monomial $m(\Delta)=m_\ll(\Delta)$ with respect to the reflection ordering $\ll$ studied in Proposition \ref{szrzs}. If $\Delta\in B(e,v)$ (or $\Delta\in \overline{B(e,v)}$) we consider the monomial $m_{\ll_s}(\Delta)$ with respect to the lower $s$-conjugate $\ll_s$ of $\ll$. 
With this convention we will show that the bijection $\sigma$ has the following properties:
\begin{enumerate}
\item if $\Delta\in B(e,svs)$ then $m(\sigma(\Delta))=m(\Delta)$;
\item if $ \Delta\in B(e,v)$ then $m(\sigma(\Delta))=ab\cdot m(\Delta)$;
\item if $ \Delta\in \overline{B(e,v)}$ then $m(\sigma(\Delta))=ba\cdot m(\Delta)$.
\end{enumerate}

Consider the Bruhat graph of $[e,svs]$: the vertices of this graph can be visualized as in Figure \ref{xs}, where the four shaded regions correspond respectively from left to right to: (1) elements of the form $sxs$, for some $x\leq v$; (2) elements of the form $sx$ for some $x\leq v$; (3) elements of the form $xs$ for some $x\leq v$; (4) elements smaller than or equal to $v$.  
The bijection $\sigma$ is defined as follows. Let $\Delta \in B(e,svs)$. If the smallest element in the path $\Delta$ 
which is strictly greater than $s$ is of the form $sxs$ for some $x\leq v$  
then by the Exchange Condition (see, e.g., \cite[Theorem 1.4.3]{BjBr}), $sxs \in T$ and $\Delta$ is necessarily of the form
$\Delta=(x_0\stackrel{st_1s}{\longrightarrow} sx_1s\stackrel{st_2s}{\longrightarrow}\cdots \stackrel{st_{r}s}{\longrightarrow}sx_{r}s)$, with $t_i\in P$ and $x_i \leq v$ for all $i\in [r]$ (see Figure \ref{xs}, left), and we define $\sigma(\Delta)=(x_0\stackrel{t_1}{\longrightarrow} x_1\stackrel{t_2}{\longrightarrow}\cdots \stackrel{t_{r}}{\longrightarrow}x_{r})\in B(e,v)$; 
since $st_is\ll st_{i+1}s$ if and only if $t_i\ll_s t_{i+1}$ we clearly have $m_{\ll_s}(\sigma(\Delta))=m_{\ll}(\Delta)$.

Suppose now that the smallest element in the path $\Delta$ which is strictly greater than $s$ is of the form $xs$ for some $x\leq v$ (see Figure \ref{xs}, right). 

\begin{figure} \psfrag{e}{$e$} \psfrag{vs}{$vs $}\psfrag{sv}{$sv $}\psfrag{svs}{$svs $}\psfrag{v}{$v $}\psfrag{xs}{$xs $}\psfrag{sxs}{$sxs $}\psfrag{x}{$x$}\psfrag{xks}{$x_ks$}\psfrag{sxks}{$sx_ks$}
\includegraphics[scale=.7]{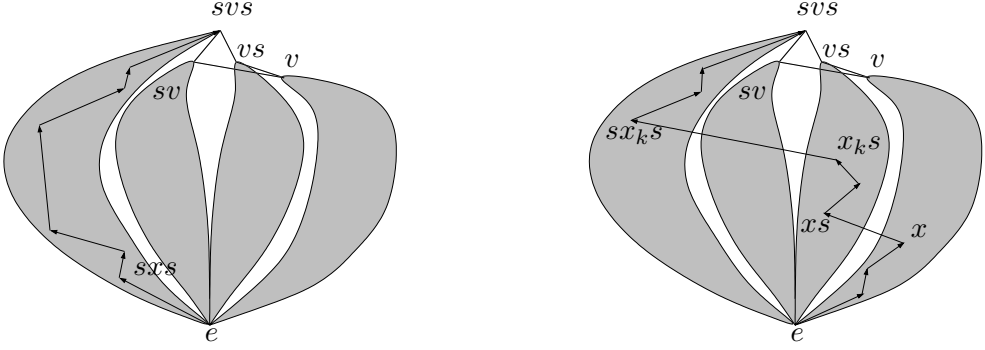}
\caption{Paths in the Bruhat graph of $[e,svs]$}
\label{xs}
\end{figure}
Then $\Delta $ is of the form
\begin{align*}
   \Delta=(x_0\stackrel{t_1}{\longrightarrow}&\cdots \stackrel{t_{i-1}}{\longrightarrow}x_{i-1}\stackrel{s}{\longrightarrow}x_{i-1}s\\
& \stackrel{st_is}{\longrightarrow}x_{i}s\cdots \stackrel{st_{k}s}{\longrightarrow}x_{k}s\stackrel{sx_k^{-1}sx_ks}{\longrightarrow}sx_{k}s\stackrel{st_{k+1}s}{\longrightarrow}\cdots \stackrel{st_{r}s}{\longrightarrow}sx_{r}s)
\end{align*}
for some integers $i,k$ such that $r \geq k\geq i-1 \geq 0$, $k \geq 1$, where $t_1,\ldots,t_{r}\in P$, $x_1,\ldots,x_{r}\leq v$. 
In this case we define $\sigma(\Delta) \in B(e,rvs)$ essentially by replacing the letter $s$ ``on the left'' by $r$. 
More precisely we let 
\begin{align*}\sigma(\Delta)=
 &(x_0\stackrel{t_1}{\longrightarrow}\cdots \stackrel{t_{i-1}}{\longrightarrow}x_{i-1}\stackrel{s}{\longrightarrow}x_{i-1}s\\
&\stackrel{st_is}{\longrightarrow}x_{i}s\cdots \stackrel{st_{k}s}{\longrightarrow}x_{k}s\stackrel{sx_k^{-1}rx_ks}{\longrightarrow}rx_{k}s\stackrel{st_{k+1}s}{\longrightarrow}\cdots \stackrel{st_{r}s}{\longrightarrow}rx_{r}s)  
\end{align*}
and it follows from Proposition \ref{szrzs} that  $m_\ll(\sigma(\Delta))=m_\ll(\Delta)$ (we observe here that if $\ell(x_k)=1$ then $sx_k^{-1}sx_ks=x_k$ and in particular we still have $sx_k^{-1}sx_ks\ll sts$ for all reflections $t\in P$).

Finally, if the smallest element strictly greater than $s$ in the path $\Delta$ is of the form $sx$ for some $x\leq v$, then $\Delta $ is of the form
\begin{align*}
\Delta=   (x_0\stackrel{t_1}{\longrightarrow}&\cdots \stackrel{t_{i-1}}{\longrightarrow}x_{i-1}\stackrel{x_{i-1}^{-1}sx_{i-1}}{\longrightarrow}sx_{i-1}\\
&\stackrel{t_i}{\longrightarrow}sx_{i}\cdots \stackrel{t_{k}}{\longrightarrow}sx_{k}\stackrel{s}{\longrightarrow}sx_{k}s\stackrel{st_{k+1}s}{\longrightarrow}\cdots \stackrel{st_{r}s}{\longrightarrow}sx_{r}s)
\end{align*}
for some integers $i,k$ such that $r \geq k\geq i-1 \geq 0$, $k \geq 1$, where $t_1,\ldots,t_{r}\in P$, $x_1,\ldots,x_{r}\leq v$, and we let $\sigma(\Delta) \in B(e,rvs)$ be defined by 
\begin{align*}
  \sigma(\Delta) =(x_0\stackrel{t_1}{\longrightarrow}&\cdots \stackrel{t_{i-1}}{\longrightarrow}x_{i-1}\stackrel{x_{i-1}^{-1}rx_{i-1}}{\longrightarrow}rx_{i-1}\\
&\stackrel{t_i}{\longrightarrow}rx_{i}\cdots \stackrel{t_{k}}{\longrightarrow}rx_{k}\stackrel{s}{\longrightarrow}rx_{k}s\stackrel{st_{k+1}s}{\longrightarrow}\cdots \stackrel{st_{r}s}{\longrightarrow}rx_{r}s).
\end{align*}
Also in this case it follows from Proposition \ref{szrzs} that $m_\ll(\sigma(\Delta))=m_\ll(\Delta)$. 
We have considered in this way all paths in $B(e,svs)$ and we have obtained all paths in $B(e,v)$ and all paths in $B(e,rvs)$ except those passing through $rs$. 

If $\Delta\in B(e,v)$, with
\[
   \Delta=(x_0\stackrel{t_1}{\longrightarrow} x_1\stackrel{t_2}{\longrightarrow}\cdots \stackrel{t_{r+1}}{\longrightarrow}x_{r+1})
\]
then we let 
\[
   \sigma(\Delta)=(x_0\stackrel{r}{\longrightarrow}r\stackrel{s}{\longrightarrow}rs\stackrel{st_1s}{\longrightarrow} rx_1s\stackrel{st_2s}{\longrightarrow}\cdots \stackrel{st_{r+1}s}{\longrightarrow}rx_{r+1}s)
\]
and if the same path $\Delta$ is considered in $\overline {B(e,v)}$ we let 
\[
  \sigma(\Delta)=(x_0\stackrel{s}{\longrightarrow}s\stackrel{srs}{\longrightarrow}rs\stackrel{st_1s}{\longrightarrow} rx_1s\stackrel{st_2s}{\longrightarrow}\cdots \stackrel{st_{r+1}s}{\longrightarrow}rx_{r+1}s).
\]
In the first case we have $m_{\ll}(\sigma(\Delta))=ab \cdot m_{\ll_s}(\Delta)$ by Proposition \ref{szrzs}. In the second case we similarly have $m_{\ll}(\sigma(\Delta))=ba \cdot m_{\ll_s}(\Delta)$ as $srs\ll st_1s$ by Proposition \ref{szrzs} used with $w=e$. 
\end{proof}

\subsection{ Homogeneous components and linear relations}
Following \cite{Rea} we consider a set $W_n$ of elements in the 3-complete Coxeter group $W$ of rank $n+1$ generated by $s_1,\ldots,s_{n+1}$ constructed recursively in the following way: we let $W_0=\{s_1\}$, $W_1=\{s_1s_2\}$ and, for $n\geq 2$, 
\[
   W_n=\{ws_{n+1}:\, w\in W_{n-1}\}\cup \{s_{n+1} w s_{n+1}:\,w\in W_{n-2}\}.
\]
We now consider the following space of $cd$-polynomials
\[
   V_n=Span\{\tilde \Psi_{e,v}:\, v\in W_n\}.
\]
Since $\ell(v)=n+1$ for all $v\in W_n$ we deduce that $V_n$ is contained in the space of $cd$-polynomials of degree bounded by $n$.
A set of generators for $V_n$ can also be described in the following way. Let $A_0=\{1\}$, $A_1=\{c\}$ and
\[
   A_n=\{c\cdot P+P':\, P\in A_{n-1}\}\cup \{ (d-1)\cdot P:\, P\in A_{n-2}\},
\]
where for all $P\in \mathcal A$ we let $P'\eqdef G'(P)$. By Corollaries  \ref{finalpyr}, \ref{finalpyrl} and Theorem \ref{finalsvs} we have that $A_n$ is a spanning set for $V_n$. We observe that $|A_n|=f_{n+1}$ and we denote its elements by $P_{n,1},\ldots,P_{n,f_{n+1}}$ in the following way. We let $P_{0,1}=1$, $P_{1,1}=c$ and 
\[P_{n,j}=
\begin{cases}
cP_{n-1,j}+P_{n-1,j}'&\textrm{if }1 \leq j\leq f_n\\ 
(d-1)P_{n-2,j-f_n}&\textrm{if }f_n <j \leq f_n + f_{n-1}
 \end{cases}
\]

The next result follows immediately from the above recursion.
\begin{lem}
\label{signs}
   Let $P_{n,j}=\sum_M a_M M$, the sum being over all monomials of degree at most $n$ (and of the same parity as $n$). 
If $M$ is a monomial of degree $n-2i$ ($i  \geq 0$) then $a_{M}(-1)^i \geq 0$.
\end{lem}

We consider the lexicographic order $\prec$ on the set of $cd$-monomials of degree $n$ for all $n\in \mathbb N$, where we let $c\prec d$. So for example, if $n=4$ we have $c^4\prec c^2d\prec cdc\prec dc^2\prec d^2$. 
The proof of the following result is a simple 
verification, and is left to the reader.
\begin{lem}
\label{deriv}
Let $M,I$ be $cd$-monomials of the same degree such that $I \preceq M$. Then the $cd$-polynomial $M'$ is a sum of monomials which are all $\succeq cI$.
\end{lem}

If $P$ is a $cd$-polynomial with non-zero homogeneous component of degree $n$, we call the minimum monomial of degree $n$ appearing in $P$ with non-zero coefficient the $n$-th initial term of $P$. We denote by $M_{n,j}$ the $n$-th initial term of $P_{n,j}$.
\begin{lem}\label{init}
   For all $n\in\mathbb N$ we have $M_{n,1}\prec M_{n,2}\prec \cdots\prec M_{n,f_{n+1}}$.
\end{lem}
\begin{proof}
  This follows from Lemma \ref{deriv} and the observation that if $P$ is a polynomial of degree $n$ and $M$ is the $n$-th initial term of $P$, then $cM$ is the $n+1$-st initial term of $cP$ and $dM$ is the $n+2$-nd initial term of $(d-1)P$. 
\end{proof}

\begin{lem}\label{structuresigns}
   Let $P_{n,j}=\sum_M a_M M$, the sum being over all monomials of degree at most $n$ (and of the same parity as $n$). If $M_0$ is a monomial of degree $n-2i$, with $i>0$, and $a_{M_0}\neq 0$ then there exists a monomial $\tilde M_0$ of degree $n-2i+2$ with $a_{\tilde M_0}\neq 0$ such that $M_0$ is obtained from $\tilde M_0$ be deleting a letter $d$. 
\end{lem}
\begin{proof}
We proceed by induction on $n$, the cases $n=0,1$ being empty.
We consider the two cases:
\begin{enumerate}
 \item if $P_{n,j}=cQ+Q'$ for some $Q\in A_{n-1}$ we let $Q=\sum b_m m$, the sum being over monomials $m$ of degree bounded by $n-1$ and of the same parity as $n-1$. 
The monomial $M_0$ will appear as a summand in $cm_0+m_0'$ for some $m_0$ such that $\deg(m_0)=n-1-2i$ and $b_{m_0}\neq 0$. 
By induction there exists $\tilde m_0$ of degree $n+1-2i$ such that $b_{\tilde m_0}\neq 0$ and such that $m_0$ is obtained from $\tilde m_0$ by deleting a letter $d$. 
Then it is not hard to see that in $c\tilde m_0+\tilde m_0'$ there is a monomial obtained by inserting a letter $d$ in $M_0$. Since, by Lemma \ref{signs}, all monomials of the same degree appearing in $Q$ have coefficients with the same sign, there cannot be cancellations when expanding $cQ+Q'$ and therefore we necessarily have $a_{\tilde M_0}\neq 0$. 
\item $P=(d-1)Q$ for some $Q\in A_{n-2}$. This is similar and simpler and is left to the reader.
\end{enumerate}
\end{proof}

We can now prove the main result of this section.

\begin{thm}\label{mainhomo}
 Let $k,n\in \mathbb N$. Then the homogeneous parts of degree $n$ of the polynomials $(d-1)^k P_{n,j}$, for $j\in[f_{n+1}]$, are linearly independent.
\end{thm}
\begin{proof}
By Lemma \ref{init}, the result will follow if we show that the initial term of the homogeneous part of degree $n$ of $(d-1)^k P_{n,j}$ equals the initial term $M_{n,j}$ of the homogeneous part of degree $n$ of $P_{n,j}$.

We need the following notation: if $M$ is a monomial of degree $n$ we let $i(M)=\max\{i\in \mathbb N: M=d^i\cdot m\textrm{ for some monomial }m\}$ and for all $j\leq i(M)$ we let $M^{(j)}$ be the monomial obtained from $M$ be deleting its first $j$ factors so $M=d^j M^{(j)}$. For example, if $M=d^2cd$ then $i(M)=2$, $M^{(0)}=M$, $M^{(1)}=dcd$ and $M^{(2)}=cd$. 

Let $P_{n,j}=\sum_M a_M M$ and $(d-1)^kP_{n,j}=\sum_M b_M M$. Then, for every monomial $M$, $deg(M) \leq n+2k$, we have
\begin{equation}
\label{binom}
   b_M=\sum_{j=0}^{\min(i(M),k)}(-1)^{k-j}\binom{k}{j}a_{M^{(j)}}.
\end{equation}
If $M$ has degree $n$ we have that $(-1)^ja_{M^{(j)}}\geq 0$ for all $j \geq 0$ by Lemma \ref{signs} and in particular we have that $b_M\neq0$ if $a_M\neq0$. In particular $b_{M_{n,j}}\neq 0$.
Now we have to show that if $M_0$ is a monomial of degree $n$ such that $b_{M_0}\neq 0$ then $M_{n,j}\prec M_0$. It follows from (\ref{binom}) that $a_{M_0^{(j)}}\neq 0$ for some $0 \leq j \leq \min(i(M),k)$. 
Repeated applications of Lemma \ref{init} imply that there exists a monomial $\tilde M$ of degree $n$, with $a_{\tilde M}\neq 0$ such that $M_0^{(j)}$ can be obtained by deleting $j$ factors $d$ from $\tilde M$. 
Therefore $M_0$ can be obtained from $\tilde M$ by moving some factors $d$ to the left and so $\tilde M\prec M_0$; 
finally, $a_{\tilde M}\neq 0$ implies $M_{n,j}\prec \tilde M$, completing the proof.
 \end{proof}
 
The following consequence of Theorem \ref{mainhomo} is the main motivation for the results in this section.

\begin{cor}
\label{norel}
   Let $n,k\in \mathbb N$. Let $a_T\in \mathbb Q$, $T\in \two^n_s$ be such that 
\[
   \sum_{T\in \twoindex^n_s }a_T b(e,v)_T=0
\]
for all Coxeter groups $W$ and all $v\in W$ such that $\ell(v)=n+2k$. Then $a_T=0$ for all $T\in \two^n_s$. 
\end{cor}
\begin{proof}If $P$ is a $cd$-polynomial let $P^{(n)}$ be the homogeneous component of degree $n$ of $P$.
   By Theorem \ref{mainhomo} and our definitions we have that the $cd$-polynomials $\tilde \Psi_{e,v}^{(n)}$, as $v$ ranges in $W_{n+2k+1}$, span the whole space of homogeneous $cd$-polynomials of degree $n$, which has dimension $f_{n+1}$. 
But by definition of the complete $cd$-index we have
\[
  \tilde \Psi_{e,v}^{(n)}=\sum_{E\in \twoindex^n} b(e,v)_E \mu_{E^{op}},
\]
and so the result now follows since $b(e,v)\in \mathscr B_{n}$.

\end{proof}

As an immediate consequence we obtain the following result which implies that the formula in Theorem \ref{finalmain} cannot be ``linearly'' simplified, even for lower intervals.

\begin{cor}
\label{indep}
   Let $a_T\in \mathbb Q$, $T\in \two^*_s$ be such that \[\sum_{T\in \twoindex^*_s}a_T b(e,v)_T=0\] for all Coxeter groups $W$ and $v\in W$. Then $a_T=0$ for all $T\in \two^*_s$.
\end{cor}

\subsection{Another family of complete $cd$-indices and a conjecture}

Now we want to construct another family of complete $cd$-indices.
\begin{thm}
   Let $(W,S)$ be a 3-complete Coxeter system, $e\neq v\in W$ and $s\in S$ be such that $s\not\leq v$. Then
\[
   \tilde \Psi_{s,svs}=\tilde \Psi_{e,v}\cdot c+\sum_{x\in (e,v)}\tilde \Psi_{e,x}\cdot d \cdot \tilde \Psi_{x,v}.
\]
\end{thm}
 \begin{proof}
    Consider a path $\Delta \in B(s,svs)$. Then two cases occur: either $\Delta$ is of the form $\Delta=(s\rightarrow sx_1\cdots )$ or  $\Delta=(s\rightarrow x_1s\cdots )$ for some $e\neq x_1\leq v$. Call $B_1(s,svs)$ the family of paths of the first kind and $B_2(s,svs)$ the family of paths of the second kind.
We claim that there is a bijection $B_1(s,svs)\longleftrightarrow \bigcup_{x\in (e,v]}B(e,x)\times B(x,v)$. 
Furthermore, if we consider on paths in $B_1(s,svs)$ and in $B(e,x)$ the order $\ll$ described in Lemma \ref{szszs} and on paths in $B(x,v)$ the lower $s$-conjugate $\ll_s$ of $\ll$ we claim  that if $\Delta\in B_1(s,svs)$ corresponds to $(\Delta',\Delta'')\in B(e,x)\times B(x,v)$ with $x\neq v$ then $m_\ll(\Delta)=m_\ll(\Delta')\cdot ab \cdot m_{\ll_s}(\Delta'')$ and if $(\Delta',\Delta'')\in B(e,v)\times B(v,v)$ then $m_\ll(\Delta)=m_\ll(\Delta')\cdot a$. 
The bijection is defined as follows: if $\Delta\in B_1(s,svs)$ then it is necessarily of the form
 \[\Delta=(s\stackrel{t_1}{\longrightarrow} sx_1\stackrel{t_2}{\longrightarrow}\cdots \stackrel{t_i}{\longrightarrow}{sx_i}\stackrel{s}{\longrightarrow}sx_is\stackrel{st_{i+1}s}{\longrightarrow}\cdots \stackrel{st_{r}s}{\longrightarrow}sx_{r}s).
\]
for some $1\leq i \leq r$. 
Then we define $\Delta'=(e\stackrel{t_1}{\longrightarrow} x_1\stackrel{t_2}{\longrightarrow}\cdots \stackrel{t_i}{\longrightarrow}{x_i})$ and $\Delta''=(x_i\stackrel{t_{i+1}}{\longrightarrow}\cdots \stackrel{t_{r}}{\longrightarrow}x_{r})$. The fact that this is a bijection is clear and that the monomial $m_\ll(\Delta)$ satisfies the stated properties follows from Lemma \ref{szszs} and the definition of $\ll_s$. We deduce that
\begin{align*}
   \sum_{\Delta\in B_1(s,svs)}m_\ll (\Delta)&=\sum_{\Delta'\in B(e,v)}m_\ll(\Delta')\cdot a+\sum_{x\in(e,v)}\sum_{\substack{\Delta'\in B(e,x)\\ \Delta''\in B(x,v)}}m_\ll(\Delta')\cdot ab\cdot m_{\ll_s}(\Delta'')\\
&=\tilde \Psi_{e,v}\cdot a+\sum_{x\in (e,v)}\tilde \Psi_{e,x}\cdot ab \cdot \tilde \Psi_{x,v}.
\end{align*}

We also claim that there is a bijection $B_2(s,svs)\longleftrightarrow \bigcup_{x\in (e,v]}B(e,x)\times B(x,v)$ such that if $\Delta$ corresponds to $(\Delta',\Delta'')\in B(e,x)\times B(x,v)$ with $x\neq v$ then $m_\ll(\Delta)=m_{\ll_s}(\Delta')\cdot ba \cdot m_{\ll_s}(\Delta'')$ and if $(\Delta',\Delta'')\in B(e,v)\times B(v,v)$ then $m_{\ll}(\Delta)=m_{\ll_s}(\Delta')\cdot b$. In this case, if $\Delta\in B_2(s,svs)$, then $\Delta$ is of the form
\[
   \Delta=(s\stackrel{st_1s}{\longrightarrow} x_1s\stackrel{st_2s}{\longrightarrow}\cdots \stackrel{st_is}{\longrightarrow}{x_is}\stackrel{sx_i^{-1}sx_is}{\longrightarrow}sx_is\stackrel{st_{i+1}s}{\longrightarrow}\cdots \stackrel{st_{r}s}{\longrightarrow}sx_{r}s),
\]
and we define $\Delta'=(e\stackrel{t_1}{\longrightarrow} x_1\stackrel{t_2}{\longrightarrow}\cdots \stackrel{t_i}{\longrightarrow}{x_i})$ and $\Delta''=(x_i\stackrel{t_{i+1}}{\longrightarrow}\cdots \stackrel{t_{r}}{\longrightarrow}x_{r})$.
It follows that
\[
 \sum_{\Delta\in B_2(s,svs)}m_\ll (\Delta)  =\tilde \Psi_{e,v}\cdot b+\sum_{x\in (e,v)}\tilde \Psi_{e,x}\cdot ba \cdot \tilde \Psi_{x,v},
\]
and the result follows.
\end{proof}

If $W_n$ is the subset of elements of the 3-complete Coxeter group constructed in the previous subsection, this result allows us to easily compute all the complete $cd$-indices of the Bruhat intervals $[s_{n+1},s_{n+1}vs_{n+1}]$ as $v$ ranges in $W_{n-1}$. 
This has allowed us to verify the following conjecture for $n\leq 17$.

\begin{conj}
For all $n>0$ the complete $cd$-indices of all Bruhat intervals of rank $n+1$ span the whole space of $cd$-polynomials of degree bounded by $n$ whose nonzero homogeneous components have degree of the same parity as $n$. 
\end{conj}

This conjecture implies the following one, which in turn would imply that the formula obtained in Theorem \ref{finalmain} cannot be ``linearly'' simplified, even if we content ourselves with a formula that only holds for all Bruhat intervals of a fixed rank.

\begin{conj}
Let $n>0$. Then there are no nontrivial relations of the form
\[\sum_{i\in\{n,n-2,\ldots\}}\sum_{T\in \twoindex^i_s}a_Tb(u,v)_T=0,   
\]
valid for all Coxeter groups $W$ and all $u,v\in W$ such that $\ell(v)-\ell(u)=n$.
\end{conj}

\end{document}